\documentclass[12pt]{amsart}

\usepackage{amsfonts}
\usepackage[utf8]{inputenc}
\usepackage{latexsym}
\usepackage{amscd}
\usepackage{amsmath}
\usepackage{amssymb}
\usepackage{graphicx}
\usepackage[mathscr]{eucal}
\usepackage{mathptmx}
\usepackage[ruled,vlined,noresetcount,algosection]{algorithm2e}
\usepackage[normalem]{ulem}

\usepackage{setspace}

\def\sideremark#1{\ifvmode\leavevmode\fi\vadjust{\vbox to0pt{\vss 
      \hbox to 0pt{\hskip\hsize\hskip1em           
 \vbox{\hsize2cm\tiny\raggedright\pretolerance10000
 \noindent #1\hfill}\hss}\vbox to8pt{\vfil}\vss}}} %
\usepackage{color}

\newtheorem{theo}[algocf]{Theorem}
\newtheorem{prop}[algocf]{Proposition}
\newtheorem{lem}[algocf]{Lemma}
\newtheorem{co}[algocf]{Corollary}
\theoremstyle{definition}
\newtheorem{de}[algocf]{Definition}
\newtheorem{example}[algocf]{Example}

\newtheorem{re}[algocf]{Remark}
\numberwithin{algocf}{section}

\newcommand{\gp}{\mathbb{P}}

\title[On $\delta$-sequences and surfaces at infinity]{On $\delta$-sequences and surfaces at infinity}

\author{C.~Galindo \and F. Monserrat \and C.-J. Moreno-\'Avila \and J.-J.~Moyano-Fern\'andez}
\curraddr{{\bf C. Galindo and J.-J.~Moyano-Fern\'andez }: Departamento de Matem\'aticas \&
Instituto Universitario de Matem\'aticas y Aplicaciones de Castell\'on (IMAC),
Universitat Jaume I. Campus de Riu Sec, 12071 Castell\'on, Spain.
\newline {\bf F. Monserrat}:  Departament de Matemàtica Aplicada \& Institut Universitari de Matemàtica Pura i Aplicada, Universitat Politècnica de València, 46022 València, Spain.
\newline {\bf C.-J. Moreno-Ávila}:  Departamento de Matem\'aticas, Escuela Polit\'ecnica, Universidad de Extremadura, 10003, C\'aceres, Spain}
\email{galindo@mat.uji.es \hskip 0.3cm framonde@mat.upv.es \hskip 0.3cm moyano@uji.es \newline cjmoravi@unex.es}

\thanks{The authors were partially funded by MICIU/AEI/10.13039/501100011033 and by “ERDF, UE”, grant PID2022-138906NB-C22. The third author was also supported by the Margarita Salas postdoctoral contract MGS/2021/14(UP2021-021) financed by the European Union-NextGenerationEU}

\begin{document}

\begin{abstract}
In most cases the semigroup at infinity $S$ of a curve $C$ with only one place at infinity is generated by a $\delta$-sequence. This sequence provides geometrical information on $C$ such as the dual graph of the resolution of the singularity of $C$ at infinity. Since different $\delta$-sequences can generate the same semigroup, it is an interesting problem to know the geometrical behaviour of curves $C$ sharing the same semigroup $S$. An analogous problem arises in a more general context when considering
surfaces at infinity and their $\delta$-semigroups.

We show how to construct $\delta$-sequences, and how to obtain different families that gene\-rate the same semigroup $S$, allowing us to study the geometrical content encoded by $S$.
\end{abstract}

\maketitle

\section{Introduction}

It is well-known that the projective plane $\mathbb{P}^2$ and the Hirzebruch surfaces $\mathbb{F}_{\sigma}$, $\sigma \neq 1$, over a field $k$, are the relatively minimal models for rational surfaces. Furthermore, any rational surface $X$ can be obtained by blowing-up either $\mathbb{P}^2$  or $\mathbb{F}_{\sigma}$ at a configuration $\mathcal{C}$ of infinitely near points. A particularly interesting family of rational surfaces are those obtained by blowing-up at a simple configuration over $\mathbb{P}^2$ or $\mathbb{F}_{\sigma}$; we call these surfaces simple. A configuration $\{p_i\}_{i=1}^n$ of infinitely near points over a smooth projective surface $X_0$ is \emph{simple} if $p_1 = p \in X_0$ and $p_i$, $i >1$, belongs to the exceptional divisor created by blowing-up at $p_{i-1}$. Thus, a \emph{simple rational surface} $X$ can be obtained from a finite sequence of point blowups as follows:
\begin{equation}\label{finiteseq}
 X = X_{n} \rightarrow \cdots \rightarrow X_{i+1}\rightarrow X_i\rightarrow \cdots
\rightarrow X_1 \rightarrow X_0,
\end{equation}
where $X_0$ is either $\mathbb{P}^2$ or $\mathbb{F}_{\sigma}$ and $\pi_{i+1}$ the blowup of $X_i$ at $p_{i+1}$, $0 \leq i \leq n-1$.

A main advantage of simple surfaces is that the sequence (\ref{finiteseq}) defines and  is defined by a divisorial valuation $\nu$ of (the function field of) $X_0$ \cite{Spiv}. The points $p_i$ are the centers of the valuation. In recent years there has been significant activity in the study of plane valuations and their Poincar\'e series (see, for instance \cite{MR4710207-cam1, MR3695804-cam2, MR3505134.cam3}).

Divisorial valuations as above admit a large subfamily which is formed by the so-called non-positive at infinity (divisorial) valuations. These valuations allow us to explicitly des\-cribe, within their context, several interesting geometrical tools that are not well-known in general. This is the case for the Seshadri-type constants \cite{GalMon, MR4650432-gal}, which were introduced in \cite{CutEinLaz}, and the Newton-Okounkov bodies \cite{GalMonMoyNic2, MR4650432-gal}, introduced in \cite{Oko3} and independently developed in \cite{LazMus, KavKho}. Furthermore, non-positive at infinity valuations are useful for providing evidence for the Nagata conjecture \cite{GalMonMoy}.

Surfaces $X$ defined by non-positive at infinity valuations posses remarkable geometrical properties. One of the most significant is that {\it the cone of curves \rm{NE}$(X)$ of $X$ is (finite) polyhedral and minimally generated}. Conversely, any simple rational surface with this last property arises from a sequence as in (\ref{finiteseq}) given by a non-positive at infinity valuation (see \cite{GalMon, GalMonMor}). The cone of curves of a variety is an important tool in global geometry and the Minimal Model Program. Yet it remains largely unknown even in the two-dimensional case \cite{Laz1}.

A \emph{non-positive at infinity valuation} $\nu$ is characterized by the fact that $\nu(f) \leq 0$ for all $f \in \mathcal{O}(Y) \setminus\{0\}$, $Y$ being $\mathbb{P}^2\setminus L$ (respectively, $\mathbb{F}_{\sigma}\setminus \{F,M_0\}$) whenever $X_0 =\mathbb{P}^2$ (respectively, $X_0 = \mathbb{F}_{\sigma}$). Here, $L$ (respectively, $F$ and $M_0$) is the line at infinity on $\mathbb{P}^2$ (respectively, are a fibre and the special section on $\mathbb{F}_{\sigma}$). Accordingly, we define the {\it semigroup at infinity of a non-positive at infinity valuation $\nu$} as
\begin{equation}
\label{SINF}
S_{\nu,\infty} := \left\{ -\nu (f) : f \in \mathcal{O}(Y)\setminus\{0\} \right\}.
\end{equation}
Therefore, a \emph{natural question} is to study simple surfaces with minimally generated polyhedral cone of curves attending its semigroup at infinity.

The literature does not currently provide information on the generators of $S_{\nu,\infty}$; however, the concept of $\delta$-sequence and that of curve with only one place at infinity are useful when  restricting  to a specific class of non-positive at infinity valuations, which give rise to what we call {\it surfaces at infinity}. In this paper, we address the aforementioned natural question  for this last class of surfaces; our findings suggests that the natural question in its entirety remains a significant challenge.

A curve $C$ in $\mathbb{P}^2$ with only one place at infinity (see Subsection \ref{sect:curves} for the definition) determines a semigroup at infinity $S_{C,\infty}$ (see Definition \ref{def:semigroupatinfinity}). This semigroup was studied by Abhyankar and Moh, who proved that, except for certain cases of the characteristic of $k$, it is generated by a $\delta$-sequence, see Definition \ref{delta} and Theorem \ref{theo:gene}. A surface at infinity is a simple surface $X$ defined by a non-positive at infinity valuation whose centers constitute a configuration $\mathcal{C}=\{p_i\}_{i=1}^n$ of infinitely near points satisfying the following properties: The point $p_1 \in X_0= \mathbb{P}^2$ is the point at the line at infinity of a curve $C$ on $\mathbb{P}^2$ with only one place at infinity such that $S_{C,\infty}$ is generated by a $\delta$-sequence whose first element is the degree of $C$, see Proposition \ref{siete+}. Moreover the remaining points $p_i$, $i>1$, are infinitely near points on the unique branch at $p_1$ defined by $C$. We also assume that $n \geq m$, where $m$ is the positive integer such that, with the notation as in (\ref{finiteseq}), the composition $\pi_m \circ \cdots \circ \pi_1: X_m \rightarrow \mathbb{P}^2$ gives rise to the minimal embedded resolution of $C$ at $p_1$.  In the characteristic zero case, surfaces at infinity were studied in \cite{CamPilReg}.

A surface at infinity $X$ is defined by a valuation $\nu$ of $\gp^2$ (called valuation at infinity) that is non-positive at infinity. Consequently, NE$(X)$ is generated by the classes in Pic$_{\mathbb{R}}(X):=\mathrm{Pic}(X) \otimes \mathbb{R}$ of the divisors $\tilde{L}$, $E_1, \ldots, E_n$, where $\tilde{L}$ (respectively, $E_i$, $1 \leq i \leq n$,) is (respectively, are) the strict transform (respectively, strict transforms) on $X$ of $L$, (respectively, the exceptional divisors given by (\ref{finiteseq})), see also \cite{CamPilReg}. Throughout the paper, we refer to the classes  $[\tilde{L}]$ and $[E_i]$ in Pic$_{\mathbb{R}}(X)$  as the generators of NE$(X)$. The relative position of these divisors and the intersection products of their classes can be deduced from an (ordered) sequence of generators of the semigroup at infinity of $\nu$, referred to as {\it $\delta$-sequence of type A,} $\Delta_A$, (see page \pageref{tipoA} for the definition). This sequence consists of a $\delta$-sequence $\Delta = \{\delta_0, \ldots, \delta_g\}$ ---which generates $S_{C,\infty}$ and  is called the $\delta$-sequence of $X$--- plus a non-negative integer $\delta_{g+1}$. The dual graph of $\nu$ schematically depicts the relative positions of the exceptional divisors. This graph, along with the number $\eta_X$ of points in $\mathcal{C}$ through which (the strict transforms of) $L$ passes, allows us to compute the intersection products of the generators of $\mathrm{NE}(X)$, see Example \ref{ex:new}. The sequence $\Delta_A$ determines the dual graph $\Gamma_\nu$ of $\nu$ and $\eta_X$ can be derived from $\delta_0$ and $\Gamma_\nu$; the $\delta$-sequence of type A completely determines the intersection matrix of the generating classes  of $\mathrm{NE}(X)$. The sequence $\Delta$ contains the most relevant information, as $\delta_{g+1}$ simply indicates the number of free points in the corresponding configuration $\mathcal{C}$ appearing after the embedded resolution of the singularity of $C$ at the point at infinity \cite{GalMonCod1}. The semigroup of $\mathbb{Z}_{\geq 0}$ generated by $\Delta$ is named the \emph{$\delta$-semigroup of $X$} and is denoted by $S_\Delta$.

For both curves with only one place at infinity $C$ and surfaces at infinity $X$, there is a $\delta$-sequence generating $S_{C,\infty}$ and $S_\Delta$. 
A $\delta$-sequence is not necessarily a minimal set of gene\-rators of the semigroups $S_{C,\infty}$ and $S_\Delta$, represented by $S$,
and its specific ordering is essential for recovering the underlying geometry. Consequently, the above semigroups $S$ encode significant information making it worthwhile to investigate the different $\delta$-sequences gene\-rating $S$.

Within the above geometrical framework, this paper provides an arithmetic study on the construction of $\delta$-sequences and how to derive different $\delta$-sequences generating the same semigroup. Let us be more explicit.

While the definition of $\delta$-sequence is descriptive, it does not provide a direct method of construction.
\emph{Section \ref{section:gluing}} is devoted to establishing a constructive approach. Theorem \ref{thm:gluing_delta} \emph{gives an easy-to-apply procedure}, supported in previous results on semigroups, \emph{that allows for giving any $\delta$-sequence} (Corollary \ref{el35}).

Let $\mathcal{S}$ be the family of surfaces at infinity sharing the same prefixed $\delta$-semigroup. For each $X \in \mathcal{S}$, $\Delta_X$ denotes its $\delta$-sequence. Given $X, X' \in \mathcal{S}$, we say that $X$ is less intricate than $X'$ whenever $\# \Delta_X < \# \Delta_{X'}$; where $\# $ denotes cardinality. This concept gives a way to organize surfaces at infinity with the same $\delta$-semigroup. An analogous definition and procedure can be formulated for the set of curves with only one place at infinity sharing the same semigroup at infinity.

Let $S$ be a semigroup generated by a $\delta$-sequence $\Delta$. Section \ref{Sec:minimal_sec} investigates \emph{how to obtain shorter $\delta$-sequences generating $S$}, should they exist;  \emph{Theorem \ref{prop:remove_delta}} is our main result in this direction, thereby allowing us obtain surfaces at infinity (or  curves with only one place at infinity) less intricate than one initially considered while preserving the same semigroup. In this context, we introduce and study the notions of MG$\delta$\emph{-sequence} and \emph{primitive} $\delta$-sequence.

Section \ref{sec:Nested_delta_sequences} considers the inverse problem. Starting with a $\delta$-sequence $\Delta$ we seek for refinements ($\delta$-sequences which add elements to $\Delta$) which generate the same semigroup. Geometrically, this corresponds to surfaces at infinity (or  curves with only one place at infinity) that are more intricate than the surface associated with $\Delta$, yet share the same semigroup. Our main results here are \emph{Theorem \ref{thm:introducir_beta}}, which characterizes refinements of order one, and \emph{Corollary \ref{t23}} that shows the existence of finitely many refinements keeping the same $\delta$-semigroup provided that one excludes multiples of the first element of the $\delta$-sequence. Finally, in \emph{Example \ref{ejemplo57}}, we begin with an MG$\delta$-sequence $\Delta_1$ (generating a $\delta$-semigroup $S$) and construct a family of its refinements. As previously noted, each resulting $\delta$-sequence determines the behavior of the generators of the cone of curves of the corresponding surfaces at infinity in the class $\mathcal{S}$ associated with $S$.
\medskip

Along the paper $\mathbb{N}$ stands for the set of non-negative integers. For $n\in\mathbb{N}$, $\mathbb{N}_{>n}$ is the infinite set $\{n+1,n+2,\ldots\}\subseteq \mathbb{N}$. Moreover, $[n]:=\{1,\ldots , n\}$ and $[0,n]:=\{0,1,\ldots , n\}$.

\section{Global geometry and $\delta$-sequences} \label{sect2a}

We are interested in the notion of $\delta$-sequence and its influence on geometric aspects related to surfaces and curves. We devote this section to introduce the above concepts.

\subsection{Curves with only one place at infinity and $\delta$-sequences}\label{sect:curves}
Let $k$ be a field and denote by $\overline{k}$ its algebraic closure. Set $\gp^2=\gp_k^2$ the projective plane over $k$.
Let $L$ be the line at infinity in the compactification of the affine plane $\mathbb{A}_k^2$ to $\gp^2$. An absolutely irreducible curve $C$ in $\gp^2$ (i.e.,
irreducible as a curve in $\gp^2_{\overline{k}}$) is said to {\it have only one place at infinity} if the intersection $C\cap L$ is a single point $p$,  $C$ has only one  branch at $p$, and, both $p$ and the branch, are rational (that is, defined over $k$). The point $p$ is then said to be the point at infinity of $C$.
\medskip

Abhyankar and Moh introduced the semigroup at infinity of a curve $C$ with only one place at infinity. They proved that a convenient system of generators of this semigroup describes the geometry of $C$ around the point at infinity. Let $F$ be the quotient field of the local ring $\mathcal{O}_{C,p}$ and consider the discrete valuation $\nu_{C,p}$  defined by $C$ on $F$. Next, we define the semigroup at infinity and the value semigroup of $C$ at $p$.

\begin{de}\label{def:semigroupatinfinity}
Let $C$ be a curve with only one place at infinity. Set $C\cap L = \{p\}$. The {\it semigroup at infinity} of $C$ is the additive sub-semigroup:
\[
S_{C,\infty}:= \left\{ - \nu_{C,p} (h) : h \in \mathcal{O}_C(C\setminus\{p\})\setminus\{0\} \right\}.
\]
\medskip
Moreover, the value semigroup of $C$ at $p$ is
$$
S_{C,p}=\{\nu_{C,p}(h) : h \in  \mathcal{O}_{C,p} \setminus \{0\}\}.
$$
\end{de}
\medskip
These semigroups are numerical, i.e.~additive submonoids of the mo\-noid $(\mathbb{N},+)$ with finite complement. The elements in the complement are called gaps of the semigroup.
\medskip

Let $m$ be a non-negative integer. In this paper, an $(m+1)$-tuple of positive integers $\Lambda =(r_0,\ldots, r_m)$ is referred to as a {\it sequence} in $\mathbb{N}_{> 0}$. Define the values
$$
d_i(\Lambda):=\mathrm{gcd}(r_0,\ldots, r_{i-1}), \ \ i\in [m+1].
$$
The corresponding sequence $(d_i(\Lambda))_{i=1}^{m+1}$ is called its $d$-sequence. Note that $d_i(\Lambda)>0$ for every $i\in [m+1]$. Since $d_{i+1}(\Lambda)$ divides $d_i(\Lambda)$, we set
$$
n_i(\Lambda):=\frac{d_i(\Lambda)}{d_{i+1}(\Lambda)} \in \mathbb{N}_{>0} \ \ \ \mbox{for every} \ i\in [m].
$$
The sequence $(n_i(\Lambda))_{i=1}^{m}$ is called the $n$-sequence  of $\Lambda$; it is customary to append $n_0(\Lambda):=1$. We will drop the notation $d_i(\Lambda), n_i(\Lambda)$ in favor of $d_i, n_i$ when the sequence $\Lambda$ is clear from the context.
\medskip

Now we introduce the notion of $\delta$-sequence (in $\mathbb{N}_{>0}$). It is a key concept in this paper.

\begin{de} \label{delta}
Let  $\Delta = ( \delta_i )_{i=0}^g$ be a sequence in $\mathbb{N}_{>0}$, $g>0$. We say that $\Delta$ is a {\it $\delta$-sequence in} $\mathbb{N}_{>0}$ (or simply a $\delta$-sequence) if the following conditions hold:
\begin{itemize}
\item[(a)] $d_{g+1}=1$ and $n_i>1$ for every $i \in [g]$.

\item[(b)] For $i\in [g]$, the integer $n_i\delta_i$ belongs to the semigroup $\mathbb{N}\delta_0+\mathbb{N}\delta_1+\cdots + \mathbb{N}\delta_{i-1}$ $=: \langle \delta_0, \ldots, \delta_{i-1} \rangle$.

\item[(c)] $\delta_i< n_{i-1} \delta_{i-1} $
for every $i\in [g]$.
\end{itemize}
\end{de}

\begin{re}
For convenience, $T=(1)$ is also considered a $\delta$-sequence, its $n$-sequence is $(n_0)=(1)$.
\end{re}

\begin{re}\vspace{1mm}
The notion of $\delta$-sequence was introduced in the literature with slight modifications. Sathaye and Stenerson did not require any condition for $\delta_0$ and $\delta_1$, whereas we do in Item (c), see \cite{Sat,SatSte}. The concept given in Definition \ref{delta} was called a \emph{characteristic} $\delta$-sequence by Sathaye; his terminology aligns with that given by Abhyankar \cite{Abh3} and Seidenberg \cite{Sei}. Suzuki and Fujimoto \cite{FujSuz, Suz} required the inequality $\delta_0<\delta_1$. Assi and Garc\'ia-S\'anchez \cite{AssiGar} added the restriction $\delta_1>d_2$ instead of $\delta_0>\delta_1$. In this paper, we use Definition \ref{delta} which was previously considered in \cite{GalMonCod2,GalMonCod1} (see also Pinkham \cite{Pin}, Reguera \cite{Reg}).

 Notice that Condition (b) in Definition \ref{delta} means that $\delta$-sequences are telescopic sequences in the sense of Shor \cite{Shor}. This corresponds to the notion of freeness introduced by Bertin and Carbonne \cite{beca77}.
\end{re}

Let us fix homogeneous coordinates $(X:Y:Z)$ on $\mathbb{P}^2$. Assume that the equation $Z=0$ gives the line at infinity and let $p=(1:0:0)$ be the point at infinity. Set $(x,y)$ coordinates in the chart $Z\neq 0$ and $(u=y/x, v=1/x)$ coordinates around the point at infinity. We can assume that the curve $C$ is defined by a monic polynomial $f(x,y)\in k[x]\![y]$ in the indeterminate $y$ with coefficients in $k[x]$.
\medskip

Abhyankar and Moh proved in \cite{AbhMoh} that, under suitable conditions on the characteristic of the supporting field, the semigroup $S_{C,\infty}$ is generated by a $\delta$-sequence:

\begin{theo} \label{theo:gene}
\label{21} Let $C$ be a curve having only one place at infinity. Assume that ${\rm char}(k)$ does not divide
$\gcd(-\nu_{C,p}(x),-\nu_{C,p}(y))$. Then, there exists a $\delta$-sequence that generates the semigroup $S_{C,\infty}$.
\end{theo}

Note that, for the sake of simplicity, for $h \in k[x,y]$, we write $\nu_{C,p} (h)$ instead of $\nu_{C,p} (\dot h)$, where $\dot h$ is the class in the field $F$ defined by $h$. In addition, when $p$ is not a singular point of the curve $C$, we get the $\delta$-sequence $T$.
\medskip

A converse of Theorem \ref{theo:gene}--due to Sathaye and Stenerson \cite[\S~3.2]{SatSte}--is also true, see our reformulation in Proposition \ref{siete+}. This requires to introduce a family of polynomials related to a $\delta$-sequence $\Delta$. Suppose $\Delta \neq T$, set $q_0:=x$, $q_1:=y$ and, for every $i\in [g]$,
\begin{equation}\label{misq}
q_{i+1}:= q_{i}^{n_{i}} - \lambda_i \left( \prod_{j=0}^{i-1} q_j^{a_{ij}}\right)
\end{equation}
for some arbitrary $\lambda_i \in k \setminus \{0\}$, where $n_i \delta_i = \sum_{j=0}^{i-1} a_{ij} \delta_j$ --\cite[Definition 2.4]{GalMonCod2}--. These polynomials lead to an explicit construction of a curve $C$ with only one place at infinity:

\begin{prop}[Sathaye-Stenerson] \label{siete+}
Let $\Delta \neq T$ be a $\delta$-sequence. Then, the equation $q_{g+1}=0$ defines a plane curve $C$ with only one place at infinity of degree $\delta_0$ such that $\Delta$ generates $S_{C,\infty}$ (because $\delta_0 > \delta_1$, \cite{Reg}). Moreover, the set $\{q_i\}_{i=0}^{g}$ is a family of approximates for
$C$ such that $- \nu_{C,p} (q_i)=\delta_i\;$ for every $i\in [0,g]$ (see \cite[Definition 2.3]{GalMonCod1} for the concept of approximate).
\end{prop}

Puiseux pairs are one of the classical invariants in the study of plane curve singularities, see e.g.~Zariski \cite{ZariskiModuli}. The following terminology is useful for introducing them in the case of curves with only one place at infinity.

\begin{de}
\label{noprincipal}
A $\delta$-sequence $\Delta=(\delta_i)_{i=0}^g \neq T$ is called \emph{non-principal} if $\delta_1$ does not divide $\delta_0$; otherwise it is called \emph{principal}. In addition, $\Delta$ is called \emph{non-degenerate} (respectively, \emph{degenerate}) if $\delta_0-\delta_1$ does not divide $\delta_0$ (respectively, divides $\delta_0$).
\end{de}

The concepts of non-principal and principal sequences are due to Sathaye, see e.g.~\cite{Sat}. The
distinction between non-degenerate and degenerate $\delta$-sequence helps to explain the arithmetic relation between the $\delta$-sequence of a curve $C$ with only one place at the infinity and its Puiseux pairs at the singularity $p$.
\medskip

The sequence of {\it Puiseux pairs} of a $\delta$-sequence $\Delta=(\delta_i)_{i=0}^g$ is the sequence of pairs of non-negative integers $\Big( (m_i,e_i) \Big)_{i = 0}^r$ defined as follows:

\noindent \textbf{(I)} If $\Delta$ is non-degenerate, we have $r=g-1$ and
$$
\begin{array}{lccl}
e_0:=\delta_0-\delta_1, & & &   m_0:=\delta_0, \\
e_i:=d_{i+1}, & & & m_i:=n_i\delta_i-\delta_{i+1}, \ \ \ \mbox{for every } i\in[g-1].
\end{array}
$$

\noindent \textbf{(II)} If $\Delta$ is degenerate, we have $r=g-2$ and
$$
\begin{array}{lccl}
e_0:=d_2=\delta_0-\delta_1, & & & m_0:=\delta_0+n_1\delta_1-\delta_2,\\
e_i:=d_{i+2},& & & m_i:=n_{i+1}\delta_{i+1}-\delta_{i+2}, \ \ \ \mbox{for every } i\in[g-2].
\end{array}
$$

\medskip

The {\it Puiseux exponents} $\{\beta'_i\}_{i=0}^r$ are also a well-known invariant of analytically irreducible plane curve germs  \cite{Cam}. They determine, and are determined by, the {\it dual graph} of the singularity of the germ. The dual graph can be regarded as a combinatorial description of the associated resolution process of the singularity. It shows the relative position of the exceptional divisors appearing in that process. The dual graph is a tree where each vertex represents the strict transform of an exceptional divisor and two vertices are joined whenever the corresponding divisors meet \cite[Figure 1]{GalMonCod2}. In our case, $\Delta$ determines the Puiseux pairs and therefore the dual graph of the resolution of $C$ at $p$. In fact, $\beta'_0 = 1$, $\beta'_1 = m_0/e_0$, and
\begin{equation}\label{z}
\beta'_i=1+ \frac{m_{i-1}}{e_{i-1}}
\end{equation}
for $2 \leq i \leq r$, cf.~\cite[eq.~(2.4), p.~720]{GalMonCod2}.

\subsection{Global geometry associated to $\delta$-sequences}\label{subsec:Coxring}

Let $X$ be a simple rational projective surface obtained by a composition of blowups as in \eqref{finiteseq}, where $X_0 = \mathbb{P}^2$. Denote by $\text{Pic}(X)$ the Picard group of $X$ and by $\text{Pic}_\mathbb{R}(X)$  the corresponding vector space $\text{Pic}(X)\otimes\mathbb{R}$ over the field of real numbers. Set $E_0$ a general projective line on $\gp^2$ and $\{E_i\}_{i=1}^n$ (respectively, $\{E_i^*\}_{i=1}^n$) the strict (respectively, total) transforms on the surface $X$ of the exceptional divisors $E_i$  created by the blowups $\pi_i$. For simplicity, $E_i$ also denotes the strict transform of $E_i$ on any surface $X_j$, $j>i.$ For $i\in[n]$, let $[E_i]$ and $[E_i^*]$ denote the classes of these divisors  modulo linear (or numerical) equivalence in $\text{Pic}_\mathbb{R}(X)$. It turns out that the sets $\{[E_i]\}_{i=0}^n$ and $\{[E_i^*]\}_{i=0}^n$ form two bases for the $\mathbb{R}$-vector space $\text{Pic}_\mathbb{R}(X)$, and the equalities $E_i=E_i^*-\sum_{p_j\to p_i}E_j^*$, $i > 0$, provide a  change of basis for  the space of divisors with exceptional support; here the arrow $\rightarrow$ means proximity, a concept defined below in Definition \ref{INF}.
\medskip

The sequence of blowups in \eqref{finiteseq} is linked to a divisorial valuation $\nu$ of $\gp^2$.  We write $\varphi_H$ for the germ at $p$ of a curve $H$ on $\mathbb{P}^2$. Moreover, for every $i\in [n]$, let $\varphi_i$ be an analytically irreducible germ of curve at $p$ whose strict transform on $X_i$ is transversal to $E_i$ at a non-singular point of the exceptional locus. We will also stand $(\phi,\varphi)_p$ for the intersection multiplicity at $p$ of two germs of curve at $p$, $\phi$ and $\varphi$.
\medskip

When considering a particular class of divisorial valuations, $X$ has a nice geometrical structure. Let us introduce these valuations (cf. \cite{GalMonMor3}). Let $L$ be as defined in Subsection \ref{sect:curves}.

\begin{de}
    A divisorial valuation $\nu$ of $\mathbb{P}^2$ is said to be non-positive at infinity if $\nu(f)\leq 0$ for every $f\in \mathcal{O}_{\mathbb{P}^2}(\mathbb{P}^2 \setminus L)\setminus\{0\}$.
\end{de}

The following theorem describes some properties of the surfaces defined by non-positive at infinity valuations \cite[Theorem 1]{GalMon}. Recall that NE$(X)$ denotes the cone of curves of $X$, i.e.~the convex cone in $\text{Pic}_\mathbb{R}(X)$ generated by the classes of irreducible curves on $X$.
\medskip

\begin{theo}\label{advances}
Let $\nu$ be a divisorial valuation of $\mathbb{P}^2$. The following conditions are equivalent:
\begin{enumerate}
    \item[(a)] The valuation $\nu$ is non-positive at infinity.
    \item[(b)]  The cone {\rm NE}$(X)$ is generated by the classes in $\text{\emph{Pic}}_\mathbb{R}(X)$ of the divisors $\tilde{L}$, $E_1,\ldots,$ $E_n$, where $\tilde{L}$ is the strict transform on $X$ of the line at infinity.
    \item[(c)] The divisor
    $$
(\varphi_L,\varphi_n)_p E_0^{*} -\sum_{j=1}^{n} \mathrm{mult}_{p_j} (\varphi_n) E_j^{*},
$$
is nef and effective.
\end{enumerate}
\end{theo}

From a geometric perspective, it is interesting to study the surfaces defined by non-positive at infinity valuations by considering their semigroups at infinity, see \eqref{SINF}. In this paper, we restrict our attention to a subfamily of these surfaces referred to as {\it surfaces at infinity}, which will be introduced in Definition \ref{INF}. We choose this subfamily because it constitutes the only case for which information on the associated semigroups at infinity is available. Given the close relationship between surfaces at infinity and curves with only one place at infinity, our study also examines the latter family.

\medskip

Let $C$ be a curve with only one place at infinity and let $p \in \gp^2$ be the point at infinity. Assume that $p$ is singular, $S_{C, \infty}$ is generated by a $\delta$-sequence and $\delta_0 = \deg C$, see Proposition \ref{siete+}. For $i \in \mathbb{N}$, consider a sequence of morphisms $\pi_{i+1}: X_{i+1}\rightarrow X_i$ such that $\pi_1$ is the blowup of $X_0 =\gp^2$ at $p_1 := p$ and, for $i>1$, the morphism $\pi_{i+1}$ is the blowup of $X_i$ at the unique point $p_{i+1}$ lying on the strict transform of $C$ and the exceptional divisor $E_i$ created by $\pi_{i}$. Starting from $\pi_1$, one can repeat the above procedure as many times as desired to get a simple rational surface $X$ which is given by the composition $\pi := \pi_n \circ \cdots \circ \pi_1 : X_n \rightarrow \gp^2$ cf.~\eqref{finiteseq}, where $X_0= \gp^2$. Notice that there exists a positive integer $m$ such that the composition of morphisms $X_m \rightarrow \gp^2$ gives rise to the minimal embedded resolution of $C$ at $p$. This means that $E_m$ is the first exceptional divisor such that the strict transform of $C$ in $X_m$ is regular and transversal to $E_m$.

\begin{de}
\label{INF}
A {\it surface at infinity} is a surface $X=X_n$ as in the above paragraph where $n \geq m$ and its cone of curves is polyhedral and minimally generated. Valuations of $\gp^2$ giving rise to surfaces at infinity are called {\it valuations at infinity}. They are non-positive at infinity valuations.
\end{de}

Note that a surface at infinity $X$ is always linked to a curve with only one place at infinity $C$. Surfaces at infinity over fields of characteristic zero are studied in \cite{CamPilReg}.

As indicated previously, a surface at infinity $X$ is defined by a sequence of blowups as
\begin{equation}\label{infinite_seq}
\pi: X= X_n \rightarrow \cdots \rightarrow X_{m} \rightarrow   \cdots \rightarrow X_i\xrightarrow{\pi_i} X_{i-1}\rightarrow \cdots \rightarrow X_1 \xrightarrow{\pi_1} X_0:=\gp^2.
\end{equation}

Let $K$ be the function field of $\gp^2$. Let $\nu$ the divisorial valuation at infinity of $K$ (centered at $\mathcal{O}_{\mathbb{P}^2,p}$) associated to (\ref{infinite_seq}) and consider its configuration of infinitely near points $\mathcal{C}_\nu=\{p_i\}_{i = 1}^n$. For $i>j$, we say that $p_i$ is {\it proximate} to $p_j$, and we write it $p_i\to p_j$, if $p_i\in E_j$, where, by abuse of notation and as above said, $E_j$ denotes the exceptional divisor created by blowing up at $p_j$ or any of its strict transforms. In addition, we say that a blowup center $p_i$ is {\it satellite} if there exists $j<i-1$ satisfying $p_i\to p_j$; otherwise, $p_i$ is said to be {\it free}. For each $i\geq 1$, denote by $\mathfrak{m}_i$ the maximal ideal of the local ring $R_i=\mathcal{O}_{X_{i-1},p_i}$, and set $\nu(\mathfrak{m}_i):=\min\{\nu(x) \ : \ x\in\mathfrak{m}_i\setminus\{0\}\}$. The values $\nu(\mathfrak{m}_i)$ satisfy the \emph{proximity equalities}, namely $\nu(\mathfrak{m}_i) = \sum_{p_j\to p_i}\nu(\mathfrak{m}_j), \ \ i\geq 1,$ whenever $\{p_j\in\mathcal{C}_\nu \ : \ p_j\to p_i\}\neq \emptyset$, see e.g. \cite[Theorem 8.1.7]{Cas}.
\medskip

Recall that, by \cite[Theorem 7.2]{Spiv}, for any $\phi \in \mathcal{O}_{\gp^2,p}$, we have
$$
\nu(\phi) = \min \{ (\phi,\varphi)_p : \varphi \mbox{ is a general element of $\nu$} \}.
$$

Let $\Delta =(\delta_0,\ldots, \delta_{g})$ be the $\delta$-sequence generating $S_{C,\infty}$. According to \cite[Section 4]{GalMonCod1}, the semigroup at infinity of $\nu$, $S_{\nu,\infty}$, is generated by the sequence $\Lambda = (\delta_0,\ldots, \delta_{g},\delta_{g+1})$, where $\delta_{g+1} = n_g\delta_g -n + m$. This is called a {\it $\delta$-sequence of type A}. \label{tipoA} Since the valuation $\nu$ is non-positive at infinity, it follows that $n \leq n_g\delta_g + m$. As a consequence, the semigroups at infinity of the divisorial valuations considered in this paper are generated by (ordered) sequences of $g+2$ positive integers. The most significant information is contained in the first $g+1$ elements which form a $\delta$-sequence (in $\mathbb{N}_{>0}$). The last value $\delta_{g+1}$ determines only the number of consecutive free points in the last block of the configuration $\mathcal{C}_\nu$ giving rise to \eqref{infinite_seq} \cite[Remark 4.1]{GalMonCod1}. Thus, given $\Delta$, $\delta_{g+1}$ determines only the difference $n-m$. The last $n-m$ generators $[E_i]$ of the cone of curves NE$(X)$ correspond to classes of strict transforms of exceptional divisors that meet transversally at a unique point. The semigroup generated by $\Delta$ is called the $\delta$-{\it semigroup of} $\nu$ (or of $X$).
\medskip

Let $X$ be a surface at infinity given by a sequence of blowups as in (\ref{infinite_seq}). As in the case of the analytically irreducible germs of plane curve, one can consider the dual graph associated to the sequence (\ref{infinite_seq}) which, by abuse of notation, we call {\it dual graph of $X$} (see Figure 1 in \cite{GalMonMor3}). This graph consists of the dual graph of the (embedded resolution of the) germ at $p=p_1$ of the associated curve $C$ with only one place at infinity plus finitely many (may be none) vertices (and edges joining them) corresponding to the last chain of consecutive free points in $\mathcal{C}_\nu$. The dual graph can be constructed  from the Puiseux exponents of the valuation at infinity $\nu$,  see   \cite[Section 6]{Spiv} or \cite[Subsection 3.2.1]{GalMonMor3}.

Following a procedure analogous to the one described prior to Subsection \ref{subsec:Coxring}, the $\delta$-sequence of type A generating $S_{\nu, \infty}$ determines the Puiseux exponents and, consequently, the dual graph of $X$ (see the forthcoming Example \ref{ex:new}). This construction permits us to compute the intersection products of the classes generating the cone NE$(X)$. Indeed, the dual graph determines the products $[E_i]\cdot [E_j]$, $1 \leq i, j \leq n$; furthermore to compute the products $[\tilde{L}]\cdot [E_i]$, it suffices to know the number $\eta_X$ of points of $\mathcal{C}_\nu$ through which (the strict transforms of) $L$ pass. Since $\delta_0 = \deg C$, by the Noether formula \cite[Theorem 8.1.6]{Cas}, one deduces that $\eta_X= \left\lceil \frac{\delta_0}{\mathrm{mult}_p(\varphi_\nu)} \right\rceil^+$, where $\left \lceil x \right \rceil^+$ is defined as the minimum positive integer bound of $\{x\}$, $\varphi_\nu$ is a general element of $\nu$ and $\mathrm{mult}_p(\varphi_\nu)$ represents its multiplicity at $p$.

\medskip
Based on the above information, it makes sense to investigate the following problem: {\it Given a semigroup at infinity $S$ of a curve with only one place at infinity, study those rational surfaces at infinity having a $\delta$-sequence of type A, $\Lambda = (\delta'_0, \ldots, \delta'_{g'+1})$,   whose $\delta$-semigroup is $S$.} The sequence $(\delta'_0, \ldots, \delta'_{g'})$ is referred to as the $\delta$-sequence of  $\Lambda$. As we said in the introduction, we do not distinguish between surfaces at infinity sharing the same $\delta$-sequence of type A because they are associated with curves with only one place at infinity whose germs at $p$ are equisingular, and the generators of their cones of curves have the same structure.
\medskip

The challenge lies in determining how many distinct $\delta$-sequences generate $S$ and how to construct them. Once a specific $\delta$-sequence is obtained, say $\Delta =(\delta_0,\ldots, \delta_{g}),$ there exists only finitely many $\delta$-sequences of type A of the form $(\delta_0,\ldots, \delta_{g},\delta_{g+1})$. These sequences correspond to valuations at infinity that yield surfaces for which the extremal rays of their cones of curves are explicitly known.
\begin{example}\label{ex:new}
At this point, we give a very simple example to illustrate the geometric component of the problem addressed in this paper and subsequent results.

Consider the $\delta$-sequence $\Delta_1:=(\delta^1_0=12,\delta_1^1=8,\delta_2^1=9)$. One can check that $\Delta_1$ is a $\delta$-sequence from Definition \ref{delta} and $\Delta_1$ is degenerate. Let $S$ be the semigroup  generated by $\Delta_1$. With the notation after Definition \ref{def:semigroupatinfinity}, $d_1^1=12,d_2^1=4,d_3^1=1, n_1^1=3$ and $n_2^1=4$. Then, the values introduced below Definition \ref{noprincipal} are $e_0^1=4$ and $m_0^1=27$, and the set of Puiseux exponents of $\Delta_1$ is $P_{\Delta_1}=\{\beta_i'(\Delta_1)\}_{i=0}^1,$ where $\beta'_0(\Delta_1)=1$ and $\beta'_1(\Delta_1)=27/4.$ The continued fraction of $\beta'_1(\Delta_1)$ is
$$
\beta'_1(\Delta_1)=6+\dfrac{1}{1+\frac{1}{3}}.
$$
Let $X$ be a surface at infinity having the $\delta$-sequence $\Delta_1$. The surface $X$ will be given by a divisorial valua\-tion $\nu_X$ whose semigroup is associated to a $\delta$-sequence of type A, $\Lambda_X:=(12,8,9,\delta_3^X)$ for some non-negative integer $\delta_3^X$ such that $0\leq \delta_3^X\leq n_2^1\delta_2^1=36$. With the  notation above $m^X=10$ and $n^X=46-\delta_3^X.$

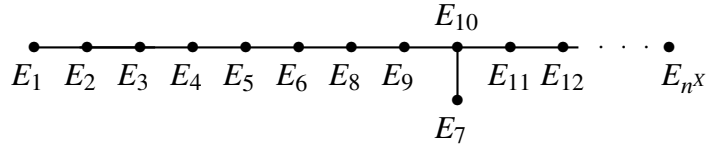
\begin{figure}[http]
\[
\unitlength=1.00mm
\begin{picture}(90.00,18.00)(-2,3)
\thicklines \put(2,15){\line(1,0){72}} \put(8,15){\line(1,0){10}}
\put(2,15){\circle*{1.5}}
\put(9,15){\circle*{1.5}} \put(16,15){\circle*{1.5}}
\put(23,15){\circle*{1.5}} \put(30,15){\circle*{1.5}}
\put(37,15){\circle*{1.5}} \put(44,15){\circle*{1.5}}
\put(51,15){\circle*{1.5}} \put(58,15){\circle*{1.5}}
\put(65,15){\circle*{1.5}} \put(72,15){\circle*{1.5}}
\put(80,15){\circle*{0.2}} \put(77,15){\circle*{0.2}}
\put(83,15){\circle*{0.2}} \put(86,15){\circle*{1.5}}
\put(58,8){\circle*{1.5}} \put(58,8){\line(0,1){7}}
\put(-1,10){$E_1$} \put(6,10){$E_2$}
\put(13,10){$E_3$} \put(20,10){$E_4$} \put(27,10){$E_5$}
\put(34,10){$E_6$} \put(41,10){$E_8$} \put(48,10){$E_9$}
\put(55,18){$E_{10}$} \put(55,3){$E_{7}$}
\put(62,10){$E_{11}$} \put(69,10){$E_{12}$}
\put(85,10){$E_{n^X}$}
\end{picture}
\]
\caption{Dual graph of $X$} \label{fig0}
\end{figure}
The relative positions of the exceptional divisors, $\{E_i^X:= E_i\}_{i=1}^{n^X}$, {\it the classes of which, together with $[\tilde{L}],$ generate the cone of curves of $X$} are displayed in the dual graph of the sequence $\pi:X:=X_{n^X}\to\cdots\to\mathbb{P}^2$ defined by $\nu_X$. See \cite[Subsection 2.3.1]{GalMonMor3} for the construction procedure (which utilizes the continued fraction of the Puiseux exponents) and Figure \ref{fig0}  for the dual graph. This allows us to easily obtain the intersection matrix of the divisors $E_i$. 
In particular, $E_i^2=-r_i-1$, where $r_i$ is the number of points of the configuration of $\nu_X$ proximate to $p_i$. Furthermore, the intersection $E_i\cdot E_j$, $i \neq j$, is either $1$ if $E_i\cap E_j\neq\emptyset$, or $0$, otherwise (see \cite[Proposition 1.1.26]{Alb}).   For instance, $E_7^2=-r_7-1=-4,$ where $r_7=3$ is the number of points of the configuration of $\nu_X$ proximate to $p_7$.

Finally, we can compute the intersection matrix of (the  classes generating) NE$(X)$, Figure \ref{fig2}. Specifically, we only need to augment the intersection matrix of the divisors $E_i$ with a first row containing the intersection products with $\tilde{L}$. In this case, the multiplicity $\mathrm{mult}_p(\varphi_{\nu_X})=4$ and, as explained previously, $\eta_X = 12/4 =3$. Hence, setting $E_i^*$ the total transforms on $X$ of the divisors $E_i$, we obtain $\tilde{L} = E_0^* - E_1^*-E_2^*-E_3^*$, which yields $\tilde{L}^2 = -2$, $\tilde{L} \cdot E_3=1$ and $\tilde{L} \cdot E_i=0$, for all $i \neq 3$.
\medskip

\begin{figure}[http]
{\small $$ \left(\begin{array}{ccccccccccccccccc}
-2 & 0  & 0  & 1  & 0  & 0  & 0  & 0  & 0  & 0  & 0  & 0 & \cdots & 0 & 0 & 0 & 0 \\
0  & -2 & 1  & 0  & 0  & 0  & 0  & 0  & 0  & 0  & 0  & 0 & \cdots & 0 & 0 & 0 & 0 \\
0  & 1  & -2 & 1  & 0  & 0  & 0  & 0  & 0  & 0  & 0  & 0 & \cdots & 0 & 0 & 0 & 0 \\
1  & 0  & 1  & -2 & 1  & 0  & 0  & 0  & 0  & 0  & 0  & 0 & \cdots & 0 & 0 & 0 & 0 \\
0  & 0  & 0  & 1  & -2 & 1  & 0  & 0  & 0  & 0  & 0  & 0 & \cdots & 0 & 0 & 0 & 0 \\
0  & 0  & 0  & 0  & 1  & -2 & 1  & 0  & 0  & 0  & 0  & 0 & \cdots & 0 & 0 & 0 & 0 \\
0  & 0  & 0  & 0  & 0  & 1  & -3 & 0  & 1  & 0  & 0  & 0 & \cdots & 0 & 0 & 0 & 0 \\
0  & 0  & 0  & 0  & 0  & 0  & 0  & -4 & 0  & 0  & 1 & 0 & \cdots & 0 & 0 & 0 & 0 \\
0  & 0  & 0  & 0  & 0  & 0  & 1  & 0  & -2 & 1  & 0  & 0 & \cdots & 0 & 0 & 0 & 0 \\
0  & 0  & 0  & 0  & 0  & 0  & 0  & 0  & 1  & -2 & 1  & 0 & \cdots & 0 & 0 & 0 & 0 \\
0  & 0  & 0  & 0  & 0  & 0  & 0  & 1  & 0  & 1  & -2 & 1 & \cdots & 0 & 0 & 0 & 0 \\
\vdots & \vdots & \vdots & \vdots & \vdots & \vdots & \vdots & \vdots & \vdots & \vdots & \vdots & \vdots & \vdots & \vdots & \vdots & \vdots & \vdots \\
0 & 0 & 0 & 0 & 0 & 0 & 0 & 0 & 0 & 0 & 0 & 0 & \cdots & 1 & -2 & 1  & 0  \\
0 & 0 & 0 & 0 & 0 & 0 & 0 & 0 & 0 & 0 & 0 & 0 & \cdots & 0 & 1  & -2  & 1 \\
0 & 0 & 0 & 0 & 0 & 0 & 0 & 0 & 0 & 0 & 0 & 0 & \cdots & 0 & 0  & 1  & -1
\end{array}\right)
$$}
\caption{Intersection matrix of NE$(X)$} \label{fig2}
\end{figure}


Surfaces at infinity $Y$ whose $\delta$-sequence of type A is $\Lambda_Y=(\delta_0^2=12,\delta_1^2=8,\delta_2^2=18,\delta_3^2=9,\delta_4^Y),$ for an appropriately chosen $\delta_4^Y$, posses the same $\delta$-semigroup as $X$. Since one of our aims is to study families of surfaces at infinity sharing the same $\delta$-semigroup, $X$ and $Y$ constitute an example of the surfaces we are interested in.

With respect to $Y$, the corresponding $\delta$-sequence is $\Delta_2=(12,8,18,9)$ which also generates $S$. In this case $d_1^2=12,d_2^2=4,d_3^2=2,d_4^2=1,n_1^2=3,n_2^2=2$ and $n_3^2=2$. In addition, $e_0^2=4,m_0^2=18, e_1^2=2,m_1^2=27.$ Then $m^Y=21$ and $n^Y=57-\delta_4^Y$ where $
0\leq\delta_4^Y\leq n_3^2\delta_3^2=18$. Here the set of Puiseux exponents is   $P_{\Delta_2}=\{\beta_i'(\Delta_2)\}_{i=0}^2,$
$$
 \beta_0'(\Delta_2)=1,\,\beta_1'(\Delta_2)=\dfrac{18}{4}=4+\dfrac{1}{2} \text{ and }\beta_2'(\Delta_2)=1+\dfrac{27}{2}=14+\dfrac{1}{2}.
$$
Consequently,  {\it the generators of the cone of curves $\operatorname{NE}(Y)$ differ significantly from those of $\operatorname{NE}(X)$, yet, as noted, $X$ and $Y$ share the same $\delta$-semigroup $S$.} We have precise information regarding these generators as the relative position and the intersection matrix of the exceptional divisors $\{E_i^Y:= E'_i\}_{i=1}^{n^Y}$ whose classes (together with $[\tilde{L}]$) generate the cone $\operatorname{NE}(Y)$ can be schematically displayed by the dual graph of $Y$ (see Figure \ref{fig1}). The intersection matrix of NE$(Y)$ can be also obtained using the same method as for $X$; for instance $(E'_4)^2 = - 3$ and $\tilde{L} \cdot E'_3$ is again $1$ because $\eta_Y$ is also $3$. We will remark, later, that $X$ is less intricate than $Y$.

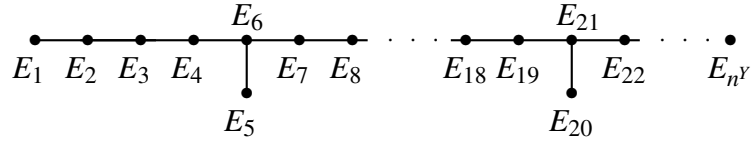
\begin{figure}[http]
\[
\unitlength=1.00mm
\begin{picture}(90.00,18.00)(0,3)
\thicklines \put(2,15){\line(1,0){44}} \put(8,15){\line(1,0){10}}
\thicklines \put(57,15){\line(1,0){25}}
\put(2,15){\circle*{1.5}}
\put(9,15){\circle*{1.5}} \put(16,15){\circle*{1.5}}
\put(23,15){\circle*{1.5}} \put(30,15){\circle*{1.5}}
\put(30,8){\circle*{1.5}}  \put(30,8){\line(0,1){7}}
\put(37,15){\circle*{1.5}} \put(44,15){\circle*{1.5}}
\put(49,15){\circle*{0.2}} \put(52,15){\circle*{0.2}}
\put(55,15){\circle*{0.2}}
\put(59,15){\circle*{1.5}} \put(66,15){\circle*{1.5}}
\put(73,15){\circle*{1.5}}
\put(73,8){\circle*{1.5}}\put(73,8){\line(0,1){7}}
 \put(80,15){\circle*{1.5}}
\put(91,15){\circle*{0.2}} \put(88,15){\circle*{0.2}}
\put(85,15){\circle*{0.2}} \put(94,15){\circle*{1.5}}
%
\put(-1,10){$E_1$} \put(6,10){$E_2$}
\put(13,10){$E_3$} \put(20,10){$E_4$}
\put(27,3){$E_5$} \put(28,17){$E_6$}
\put(34,10){$E_7$} \put(41,10){$E_8$}
\put(56,10){$E_{18}$} \put(63,10){$E_{19}$}
\put(70,3){$E_{20}$} \put(71,17){$E_{21}$}
\put(77,10){$E_{22}$} \put(91,10){$E_{n^Y}$}
%
%
%
\end{picture}
\]
\caption{Dual graph of $Y$} \label{fig1}
\end{figure}

\end{example}

In the case when char$(k)=0$, there is a nice geometrical characterization of the valuations at infinity and surfaces we are interested in. With the notation above, let $C$ be the underlying  curve at infinity given by $f(x,y)=0$ and consider $\mathcal{F}$ the pencil of plane curves on $\gp^2$ whose equations in the affine chart $Z\neq 0$ are $f(x,y)=\lambda,$ where $\lambda\in k$. By Moh \cite{Moh}, every curve in $\mathcal{F}$ has only one place at infinity and is equisingular at $p.$  Furthermore, let $r\geq m$ be the smallest positive integer such that the composition of the first $r$ blowups in a sequence as  \eqref{infinite_seq} eliminates the base points of the pencil $\mathcal{F}$. According to \cite[Remark 4.2]{GalMonCod1}, $\nu$ is a valuation at infinity whenever $n \leq r$.

\section{Constructing $\delta$-sequences}\label{section:gluing}


The classical definition of $\delta$-sequence (in $\mathbb{N}_{>0}$), Definition \ref{delta}, is of implicit type. That is, while it is easy to determine whether a given sequence in $\mathbb{N}_{>0}$ is a  $\delta$-sequence, constructing relatively large $\delta$-sequences is difficult.  In this section we state a characterization result of $\delta$-sequences and an easy to apply construction method based on the gluing of numerical semigroups. These results allow us to prove that all $\delta$-sequences can be obtained via gluing and they show how to get all $\delta$-sequences with a fixed first element $\delta_0$. Our results are algorithmic but not efficient for very large integers.
\medskip

We begin with a result that explains how to construct any $\delta$-sequence.


\begin{prop} \label{alternate}
A sequence $\Delta = (\delta_i )_{i=0}^g$ in $\mathbb{N}_{>0}$ is a $\delta$-sequence if and only if there exists two sequences in $\mathbb{N}_{>0}$,  $(a_i)_{i=0}^g$  and $(n_i)_{i=0}^g$ with $n_ i>1$ for $i \geq 1$, such that
$$
\delta_i= a_i \prod_{j=i+1}^g n_{j},
$$
where $n_0 =a_0 = 1$, $\gcd(n_i,a_i)=1$, $a_i < n_i n_{i-1} a_{i-1}$ for $ 1 \leq i \leq g$ (by convention, we set $\prod_{j=i+1}^g n_j=1$ if $i=g$) \label{A}  and, for $i \geq 2$, the integer $a_i$ must belong to the semigroup generated by the values $a_{i-1}$ and $a_{\ell}\prod_{\ell+1 \leq j \leq i-1} n_j $, for $\ell=0, 1, \ldots, i-2$.
\end{prop}

\begin{proof}
First we assume that $\Delta$ is a sequence of positive integers as in the statement. We show that $\Delta$ is a $\delta$-sequence by verifying that it fulfils Conditions (a), (b) and (c) in Definition \ref{delta}. Condition (a) is satisfied since $\gcd(a_g,n_g)=1$, which implies $d_i= n_i\cdots n_g$ and $d_{g+1}=1$. Let
$$
a_{i} = \sum_{\ell=0}^{i-2} \alpha_{\ell}\Big(a_{\ell} \prod_{j=\ell+1}^{i-1} n_j \Big)+ \alpha_{i-1} a_{i-1},
$$
where $\alpha_{\ell} \in \mathbb{N}$, $\ell \in [0, i-1]$. To see that Condition (b) holds, it suffices to multiply both sides of the equality by $n_i\cdots n_g$. Finally, multiplying  both sides of the inequality $a_i < n_i n_{i-1} a_{i-1}$ by $n_{i+1}\cdots n_g$ yields Condition (c).

Conversely, assume that $\Delta = ( \delta_i )_{i=0}^g$  is a $\delta$-sequence. Then, $d_1=\delta_0$, $d_2=\gcd(\delta_0,\delta_1)$, and by the definition of $n_1$, $\delta_0 = \gcd(\delta_0,\delta_1) n_1$. Analogously,
$$
\gcd(\delta_0,\delta_1) = \gcd(\delta_0,\delta_1,\delta_2) n_2.
$$
By iteration, we conclude that $\delta_0 = n_1\cdots n_g$. Since $d_2=\gcd(\delta_0,\delta_1) = n_2\cdots n_g$, we have that $\delta_1 = a_1 \prod_{j=2}^g n_j$ for some integer $a_1$ such that $\gcd(n_1,a_1)=1$. Repeating the same reasoning, it is straightforward to show that each $\delta_i$ can be written in the form $\delta_i= a_i \prod_{j=i+1}^g n_j$ as stated. Finally, arguments similar to those above establish the remaining stated conditions for the values $n_i$ and $a_i$.
\end{proof}

A first consequence is the following one.

\begin{co} \label{c21}
Let $\delta_0$ be a positive integer. There exist finitely many $\delta$-sequences whose first element is $\delta_0$. Moreover, the maximal value for $g$ among these $\delta$-sequences is the number of prime factors of $\delta_0$ counting repetitions.
\end{co}

\begin{proof}
The result follows from the facts that $\delta_0$ admits finitely many factorizations $\prod_{j=1}^g n_j$ and the inequalities $a_i < n_i n_{i-1} a_{i-1}$.
\end{proof}

The value $\delta_0$ of the $\delta$-sequence generating the semigroup $S_{C,\infty}$ of the curve $C$ with only one place at infinity attached to a surface at infinity coincides  with the degree of $C$. Therefore, we have proved the existence of finitely many semigroups at infinity  for curves with only one place at infinity of fixed degree. In addition, Proposition \ref{alternate} implies that the cardinality of a $\delta$-sequence cannot exceed $1$ plus the number of prime factors of $\delta_0$.
\medskip

Let $S_1$ (respectively, $S_2$) be a numerical semigroup generated by positive integers $a_1,\ldots , $ $a_m$ (respectively, $a_{m+1},\ldots , a_n$). Let $\beta\in S_1 \setminus \{a_1,\ldots , a_m\}$ and $\alpha\in S_2 \setminus \{a_{m+1},\ldots , a_n\}$ with $\alpha$ and $\beta$ being coprime. We say that the semigroup
$$
S=\langle \alpha a_1,\ldots \alpha a_m, \beta a_{m+1},\ldots, \beta a_n \rangle
$$
is the \emph{gluing of $S_1$ and $S_2$ through $\alpha$ and $\beta$}, see e.g.~\cite{RosGarsan}. This operation was introduced by Delorme \cite{delormeglue} (and, later, named gluing by Rosales and Garc\'ia-S\'anchez \cite[Chapter~8]{RosGarsan}) in order to study complete intersection semigroups. Semigroups at infinity of curves with only one place at infinity are free in the sense of Bertin and Carbonne \cite{beca77}, hence complete intersection \cite[\S 15]{delormeglue}. Next, we give conditions under which the gluing two $\delta$-sequences provides a $\delta$-sequence.

\begin{theo}\label{thm:gluing_delta}
    Let $\Delta_1=(\delta_0^1,\ldots,\delta_{g_1}^1)$ and $\Delta_2=(\delta_0^2,\ldots,\delta_{g_2}^2)$ be two $\delta$-sequences. Let $S_{\Delta_1}$ (respectively, $S_{\Delta_2}$) the semigroup spanned by $\Delta_1$ (respectively, $\Delta_2$). Consider the sequence in $\mathbb{N}_{>0}$ $\Delta=(\alpha\delta_0^1,\ldots,\alpha\delta_{g_1}^1,\beta\delta_0^2,\ldots,\beta\delta_{g_2}^2)$ generating the gluing of $S_{\Delta_1}$ and $S_{\Delta_2}$ through $\alpha$ and $\beta$, where $\beta\in S_{\Delta_1}$, $\alpha\in S_{\Delta_2}$ and $\gcd(\alpha,\beta)=1$. Assume also that $\alpha=m \delta_0^2$, $m\in\mathbb{N}_{>1}$, and $\beta$ satisfies the condition $\beta<m\cdot n_{g_1}(\Delta_1)\cdot \delta_{g_1}^1$. Then $\Delta$ is a $\delta$-sequence.
\end{theo}
\begin{proof}
Write $\Delta=(\delta_0,\ldots,\delta_{g_1+g_2+1})$ and consider its attached $d$-sequence $$d_i=\gcd(\delta_0,\ldots,\delta_i), \, i\in [g_1+g_2+2],$$ and its $n$-sequence $n_i=d_i/d_{i+1}$, $i\in [g_1+g_2+1]$. The $d$- and $n$-sequences of $\Delta$, $\Delta_1$ and $\Delta_2$ are related as follows:
\begin{equation}\label{eq:thm_gluing_deltaseq}
     \begin{array}{c}
        d_i=\alpha\,d_i(\Delta_1), \text{ for } i\in[g_1+1]; \; n_i= n_i(\Delta_1), \text{ for } i\in[g_1]; \\[2mm]
        d_i=\gcd(\alpha,d_{i-g_1-1}(\Delta_2)), \text{ for }i= g_1+2,\ldots , g_1+g_2+2;\\[2mm]  n_{g_1+1}=\dfrac{d_{g_1+1}}{d_{g_1+2}}=\dfrac{\alpha}{\gcd(\alpha,d_1(\Delta_2))},\; \text{ and } n_i= \dfrac{\gcd(\alpha,d_{i-g_1-1}(\Delta_2))}{\gcd(\alpha,d_{i-g_1}(\Delta_2))},
     \end{array}
     \end{equation}
     for $i=g_1 + 2,\ldots ,g_1+g_2+1$.

     It remains only to prove that $\Delta$ is a $\delta$-sequence. 
     Using our hypotheses, we obtain
     $$
     \begin{array}{c}
        d_i=d_{i-g_1-1}(\Delta_2), \text{ for }g_1+2\leq i\leq g_1+g_2+2,\  n_{g_1+1}=m \\[2mm] \text{ and } n_i= n_{i-g_1-1}(\Delta_2),  \text{ for } g_1 + 2\leq i\leq g_1+g_2+1.
     \end{array}
     $$

Then, since $\Delta_1$ and $\Delta_2$ are $\delta$-sequences, Condition (a), as well as Conditions (b) and (c) of Definition \ref{delta}, are satisfied for all indices except $i=g_1+1$. To finish, we verify these two last conditions for $i=g_1+1$. Indeed, regarding Condition (b), we must show that $n_{g_1+1}\delta_{g_1+1}\in\langle\delta_0,\ldots,\delta_{g_1}\rangle$ i.e.~that $\alpha\,\beta \in\langle \alpha\delta_0^1,\ldots,\alpha\delta_{g_1}^1\rangle$, which holds because $\beta\in S_{\Delta_1}$. To conclude, we verify Condition (c) by checking the inequality $\delta_{g_1+1}<n_{g_1}\delta_{g_1}$, which is true because, by hypothesis,  $\beta<m\,n_{g_1}(\Delta_1)\,\delta_{g_1}^1$.

\end{proof}

\begin{de}
Given two $\delta$-sequences $\Delta_1$ and $\Delta_2$ and two positive integers $\alpha,\beta$ as in Theorem \ref{thm:gluing_delta}, the $\delta$-sequence $\Delta$ constructed in that theorem is called the \emph{gluing of $\Delta_1$ and $\Delta_2$ through $\alpha$ and $\beta$}. We denote $\Delta$ by $\Delta_1\cup_{\alpha,\beta}\Delta_2$.
\end{de}

The $\delta$-sequence $\Delta_1$ in Example \ref{ex:new} can be obtained by considering the sequence $\Gamma_1=(3,2),$ the trivial one and gluing both through $\alpha=4$ and $\beta=9.$  Next, we present a more complete example of how to construct a $\delta$-sequence by using the gluing procedure.
\begin{example}
Consider the $\delta$-sequences $\Delta_1=(27,18,21,13)$ and $\Delta_2=( 20,8,15)$. By Theorem \ref{thm:gluing_delta}, we can get a $\delta$-sequence $\Delta$ as follows. Set $m=2$, then $\alpha=2 \cdot 20 =40$. Since $n_{g_1}(\Delta_1)=3$, we may choose $\beta=31$ and therefore
$$
\Delta= \Delta_1\cup_{40,31}\Delta_2 = (1080,720,840,520,620,248,465)
$$
is a $\delta$-sequence.
\end{example}

The following result proves that any $\delta$-sequence can be constructed via gluing.

\begin{co} \label{el35}
Let $\Delta=(\delta_0,\ldots,\delta_{g})$ be a $\delta$-sequence. Then, $\Delta$ can be computed as
\[
\Delta=\left(\ldots\left(\left(T\bigcup_{\substack{n_1,a_1}}T\right)\bigcup_{n_2,a_2}T\right)\ldots\right)\bigcup_{n_g,a_g}T,
\]
where $T$ is the trivial $\delta$-sequence and $(a_i)_{i=1}^g$ and $(n_i)_{i=1}^g$ the sequences associated to $\Delta$ given by Proposition \ref{alternate}.
\end{co}
\begin{proof}
By Proposition \ref{alternate}, the elements in $\Delta$ can be expressed as products of elements in  certain sequences of integers $(n_i)$ and $(a_i)$, $i \in [0,g]$, which satisfy  the properties stated therein. It follows that $T\cup_{n_1,a_1}T=\Delta_1=(n_1,a_1)$ is a $\delta$-sequence. Now, consider $\alpha = n_2$ and $\beta =a_2$; it is clear that $\beta \in S_{\Delta_1}$. By Theorem \ref{thm:gluing_delta}, $\Delta_1\cup_{n_2,a_2}T=\Delta_2=(n_1n_2, a_1\,n_2 , a_2)$ is a $\delta$-sequence. Finally, iterating this procedure, one gets $\Delta=(\delta_0,\ldots,\delta_{g})$ which concludes the proof.
%
\end{proof}


Based on the previous result, we present the forthcoming Algorithm \ref{algo_disjdecomp}, which allows us to construct any $\delta$-sequence. Notice that, the output depends on choices allowing for the generation of all different $\delta$-sequences.
\medskip

\RestyleAlgo{ruled}
\begin{algorithm}[h!]
\setstretch{1.}
\SetKwComment{Comment}{/* }{ */}
\SetAlgoInsideSkip{medskip}
\SetKwInOut{Input}{Input}\SetKwInOut{Output}{Output}\SetKwFunction{Length}{Length}
\SetKwFunction{Add}{Add}\SetKwFunction{NS}{in NumericalSemigroup}
\Input{$(\Delta,(\alpha,\beta)),$ where $\Delta$ is a $\delta$-sequence and $(\alpha,\beta)$ a pair of positive integers.}
\tcc{$\Delta$ can be the trivial $\delta$-sequence $T$}
\Output{$\Delta'$, which is the $\delta$-sequence that glues $\Delta$ and $T$ through $\alpha$ and $\beta$.}
\BlankLine
\If{$\alpha \leq 1$}{return False;}
\If{$\gcd(\alpha,\beta)>1$}{return False;}
\If{$\beta$ \NS($\Delta$) =$ \mathrm{False}$}{return False;}
$n:=\gcd(\delta_0,\delta_1,\ldots,\delta_{g-1})$; \\
\eIf{$\beta \geq n\, \alpha\, \delta_{g}$}{return False;}{return  $(\alpha\delta_0,\alpha\delta_1,\ldots,\alpha\delta_g,\beta)$;}
\caption{Construction of a $\delta$-sequence from a pair of positive integers}\label{algo_disjdecomp}
\end{algorithm}

\medskip




We conclude this section by presenting a method to obtain all $\delta$-sequences with a fixed first element $\delta_0$ of reasonable size.
To do so, it is sufficient to perform the following two steps.
\begin{enumerate}
\item[\textbf{(I)}] Compute all factorizations of $\delta_0$. Each of them yields a $n$-sequence.
\item[\textbf{(II)}] For each of these $n$-sequences $(n_1,n_2,\ldots , n_g)$, calculate the family of sequences $\{(a^k_1, a^k_2,\ldots , a^k_g)\}_{k=1}^{t}$ introduced in Proposition \ref{alternate}.
\end{enumerate}
Let us describe those families $\{(a^k_1, a^k_2,\ldots , a^k_g)\}_{k=1}^{t}$ corresponding to a given $n$-sequence $(n_1,n_2, \ldots , n_g)$.

\begin{enumerate}
\item[\textbf{(II.1)}] Choose $a^k_1$ such that $a^k_1<n_1$ and $\gcd(a^k_1,n_1)=1$. There are finitely many va\-lues which will determine the values $a^k_1$ whose index $k$ depends on the remaining choices.
\item[\textbf{(II.2)}] Choose $a^k_2$ such that $a^k_2\in\langle n_1, a^k_1\rangle, a^k_2<a^k_1n_1n_2$ and $\gcd(a^k_2,n_2)=1$. As above, one has finitely many choices for $a^k_2$.
\item[\textbf{(II.3)}] Iteratively, for $i=3,\ldots,g$, we choose $a^k_i$ such that
$$
\begin{array}{c}
a^k_i\in\langle n_1n_2\cdots n_{i-1}, a^k_1n_2\cdots n_{i-1},\ldots, a^k_{i-2}n_{i-1},a^k_{i-1}\rangle, \\[2mm]   a^k_i<a^k_{i-1}n_{i-1}n_{i}\text{ and }\gcd(a^k_i,n_i)=1.\end{array}
$$
\end{enumerate}

In this way, running over all the possible choices, one gets the complete finite family of $\delta$-sequences
$$
(\delta_0,\delta_1^k,\ldots , \delta_g^k) = (n_1n_2\cdots n_{g}, a^k_1n_2\cdots n_{g},\ldots, a^k_{g-1}n_{g},a^k_{g}),
$$
$k=1,2,\ldots , t$, whose first element is $\delta_0$.
\medskip

Let $X$ be a surface at infinity and denote by $\nu$ its valuation (of $\gp^2$) at infinity. As we said above, the semigroup at infinity $S_{\nu,\infty}$ is generated by a $\delta$-sequence of type A, $\Lambda = (\delta_0,\ldots,\delta_{g},\delta_{g+1})$,
$\Delta=(\delta_0,\ldots,\delta_{g})$ being {\it the $\delta$-sequence of $\nu$}. Since $\Delta$ encodes important geometrical information, we fix the numerical semigroup $S = S_\Delta$ generated by $\Delta$ ---which we call the {\it $\delta$-semigroup of $\nu$ (or of $X$)}. \emph{Our goal is to identify surfaces at infinity, with different $\delta$-sequences of type A, whose $\delta$-semigroup is $S$}. Then, from an algebraic point of view, our problem is reduced to obtain as much information as possible regarding the $\delta$-sequences that generate the same numerical semigroup $S$.
\medskip

Denote by $\mathcal{S}$ the set of surfaces at infinity sharing the same $\delta$-semigroup $S$. For $ X \in \mathcal{S}$, let $\Delta_X$ denote its $\delta$-sequence. Given $ X, X' \in \mathcal{S}$, we say that $X$ is {\it less intricate} than $X'$ if $\# \Delta_X < \# \Delta_{X'}$. We introduce this terminology to organize the surfaces within each one class $\mathcal{S}$ defined by different $\delta$-semigroups. For instance, any surface $X$ in Example \ref{ex:new} is less intricate than any surface $Y$ in the same example. Note that both surfaces have the same $\delta$-semigroup. We can analogous define {\it less intricate} for curves with only one place at infinity that share the same semigroup at infinity.


\section{Shortening $\delta$-sequences for a fixed $\delta$-semigroup} \label{Sec:minimal_sec}

In this section, we consider a semigroup $S$ as above and \emph{investigate how to shorten the $\delta$-sequences that generate $S$}. The subsequent section will address the opposite problem: starting with a $\delta$-sequence $\Delta$ generating $S$, we will search for  $\delta$-sequences with cardinality greater than that of $\Delta$, all generating the same semigroup $S$.
\medskip

When searching for $\delta$-sequences as short as possible, the ideal case occurs when the $\delta$-sequence is a minimal generating set of $S$.

\begin{de}
Let $\Delta$ be a $\delta$-sequence (in $\mathbb{N}_{>0}$) and set $S_\Delta$ be the numerical semigroup that generates. $\Delta$ is called an MG$\delta$-sequence whenever its elements constitute a minimal system of generators of $S_\Delta$.
\end{de}

\begin{re}
Consider the $\delta$-sequence  $\Delta = (30,15,12,10)$ and the semigroup that generates $S = S_\Delta$. In this case, $(15,12,10)$ is an MG$\delta$-sequence that generates $S$. However, it is not the unique MG$\delta$-sequence that generates $S$. For instance $(15,10,12)$ and $(12,10,15)$ are also MG$\delta$-sequences that generate $S$.
\end{re}

Let $\Delta$ be a $\delta$-sequence that generates a semigroup $S = S_\Delta$. To obtain $\delta$-sequences with lower cardinality than that of $\Delta$, we seek a procedure to successively remove elements from $\Delta$ such that the semigroup remains  unchanged and the condition of being $\delta$-sequence is preserved. The following result gives interesting information regarding the preservation of the semigroup.

\begin{lem}\label{lemma:prop3Sathaye}
Let $\Delta=(\delta_0,\delta_1,\ldots,\delta_g)$ be a $\delta$-sequence.
If $\delta_i$ can be written as a sum of two elements belonging to the semigroup generated by $\Delta,$  then there exists an index $j >i$  such that $\delta_i=n_j\delta_j$.
\end{lem}

\begin{proof}
See the proof of Statement (5) in \cite[page~13]{SatSte}.
\end{proof}

\begin{re}
Lemma \ref{lemma:prop3Sathaye} does not imply that, for every $\delta_j$ dividing $\delta_i$, the equality $\delta_i=n_j\delta_j$ holds. Indeed, consider the $\delta$-sequence $\Delta=(225,150,30,12,10)$. We get
$$
   \delta_1=150= 5\cdot \delta_2 = 15 \cdot \delta_4.
$$
Then $n_2=5$ as Lemma \ref{lemma:prop3Sathaye} states, but $n_4=3\neq 15$.
\end{re}

A characterization of MG$\delta$-sequences is given in our next result.

\begin{prop}
Let $\Delta=(\delta_i)_{i=0}^g$ be a $\delta$-sequence with $g\geq 2$, and keep the notation as in Proposition \ref{alternate}. Then the following statements are equivalent:
\begin{itemize}
\item[(a)] The $\delta$-sequence $\Delta$ is an {\rm MG}$\delta$-sequence.
\item[(b)] There is no indices $0\leq i< j\leq g$ such that $\delta_i=n_j\delta_j$.
\item[(c)] For all $j\in [g]$ and $0\leq i\leq j-1,$ we have $a_j\neq a_i\prod_{k=i+1}^{j-1}n_k$.
\end{itemize}
\end{prop}

\begin{proof}
First, notice that the equivalence between Items (a) and (b) follows from Lemma \ref{lemma:prop3Sathaye}. To prove the equivalence between Items (b) and (c), it suffices to show that the equality $\delta_i=n_j\delta_j$ is equivalent to the expression $a_j= a_i\prod_{k=i+1}^{j-1}n_k$. Indeed, by Proposition \ref{alternate}, it holds that
$$
\delta_i=a_i\prod_{k=i+1}^gn_k \text{ and }n_j\delta_j=n_ja_j\prod_{k=j+1}^gn_k.
$$
Therefore, $\delta_i$ and $n_j\delta_j$ share a common factor, $\prod_{k=j}^gn_k$, which completes the proof.
\end{proof}

\begin{re}
Given a $\delta$-sequence, deleting the elements that are a multiple of another element can violate the condition of being $\delta$-sequence. For instance, for the $\delta$-sequence $$\Delta=(225,150,30,12,10)$$ we cannot eliminate $30=3\cdot 10$, as the resulting sequence $(225,150,12,10)$ is not a $\delta$-sequence (nor the sequence $(225,12,10)$). However $\Delta'=(225,10,12)$ is a $\delta$-sequence and the semigroup $S=S_\Delta$ generated by $\Delta$ admits an MG$\delta$-sequence. Thus, the least intricate surfaces at infinity (respectively, curves with only one place at infinity) whose $\delta$-semigroup (respectively, semigroup at infinity) is $S$ are those associated with the $\delta$-sequence $\Delta'$.
\end{re}

\begin{re}
\label{nominimal}
Unfortunately, a semigroup $S$ spanned by a $\delta$-sequence does not always admit an MG$\delta$-sequence as a generating set. For example, $\Delta=( 6750, 6075, 1425, 950, 1215,$ $ 2852 )$ is a $\delta$-sequence where $6075$ is a multiple of $1215$, but erasing $6075$, no rearrangement of the remaining elements forms a $\delta$-sequence.
\end{re}

Remark \ref{nominimal} motivates the following definition. Recall that $S_\Delta$ denotes the semigroup generated by $\Delta$.

\begin{de}
A $\delta$-sequence $\Delta=(\delta_i)_{i=0}^g$ is said to {\it primitive} if there is no other $\delta$-sequence $\Delta'=(\delta_i')_{i=0}^{g'}$ such that $g' <g$ and $S_\Delta = S_{\Delta'}$.
\end{de}

We now propose a strategy for removing an element from a $\delta$-sequence while ensuring the resulting sequence remains a $\delta$-sequence.

\medskip

Let $\Delta=(\delta_0,\delta_1,\ldots,\delta_g)$ be a $\delta$-sequence such that $\delta_i=n_j \delta_j$ for two indices $i$ and $j$ such that  $0\leq i < j \leq g$. The forthcoming Lemma \ref{lema:alpha} and Theorem \ref{prop:remove_delta} distinguish four cases:
\begin{itemize}
\item[\textbf{(I)}] $i\geq 1$.
\item[\textbf{(II)}] $i=0,j\neq 1$ and $\delta_1<\delta_j$.
\item[\textbf{(III)}] $i=0,j\neq 1$ and $\delta_1>\delta_j$.
\item[\textbf{(IV)}] $i=0$ and $j=1$.
\end{itemize}
As Case \textbf{(IV)} covers the case when $\Delta$ is principal, in Cases \textbf{(I)}, \textbf{(II)} and \textbf{(III)} we assume that $\Delta$ is non-principal. In Cases \textbf{(I)}, \textbf{(II)} and \textbf{(IV)}, we define the sequence $\Delta '= (\delta_0', \delta_1', \ldots , \delta_{g-1}')$ as follows (as long as the subindices make sense):
\[
\delta_k'=\delta_k \mbox{ for } k=0,\ldots, i-1,i+1,  \ldots, j-1, \delta_i'=\delta_j, \mbox{ and } \delta_{k} '= \delta_{k + 1} \mbox{ for } k=j,\ldots, g-1.
\]
Otherwise (Case \textbf{(III)}), we define $\Delta'=(\delta_1, \delta_j, \delta_2, \ldots , \check{\delta}_j, \ldots, \delta_g)$, where $\check{\delta}_j$ denotes the removal of $\delta_j$. Set $(d_k)_{k=1}^{g+1}$ and $(n_k)_{k=1}^{g}$ (respectively,  $(d_k')_{k=1}^g$ and $(n_k')_{k=1}^{g-1}$) the $d(\Lambda)$ and $n(\Lambda)$-sequences (respectively, $d(\Lambda')$ and $n(\Lambda')$-sequences). The following lemma describes the relationship between these sequences under the proposed removal strategy.



\begin{lem}\label{lema:alpha}
Let $\Delta$ and $\Delta'$ be two sequences defined above. Depending on the case the $d'$-sequence and $n'$-sequence are related to the $d$-sequence and $n$-sequence as follows:

\begin{enumerate}
\item Case \textbf{(I)}:
\begin{align*}
&d_k= d_k', \ \ \ \ \ \ \ \ \mbox{for} \ \ \ k=1,\ldots , i; \\
&d_k= n_j d_k', \ \ \ \ \mbox{for} \ \  k=i+1,\ldots , j; \\
&d_{k+1}= d_{k}', \ \ \ \ \mbox{for} \ \  k=j+1,\ldots , g.\\
\end{align*}

\begin{align*}
&\mbox{For $k=1,\ldots, i-1$, we have $n_k=n_k'$;} \\
&\mbox{For $k=i$, we have $n_i'=n_j n_i$;}\\
&\mbox{For $k=i+1,\ldots, j-1$, we have $n_k=n_k'$;}\\
&\mbox{For $k=j,\ldots, g-1$, we have $n_{k +1}=n_{k}'$.}
\end{align*}

\item Case \textbf{(II)}:
\begin{align*}
&d_{k}= n_j d_{k}', \ \ \ \ \mbox{for} \ \  k=1,\ldots, j;  \\
&d_{k+1}= d_{k}', \ \ \ \ \mbox{for} \ \  k=j+1,\ldots, g.
\end{align*}

\begin{align*}
&\mbox{For $k=1,\ldots, j-1$, we have $n_k=n_k'$;}\\
&\mbox{For $k=j,\ldots, g-1$, we have $n_{k +1}=n_{k}'$.}
\end{align*}

\item Case \textbf{(III)}:
\begin{align*}
&d_{1}'= \delta_1;  \\
&d_{k}= n_j d_{k}', \ \ \ \ \mbox{for} \ \ k=2,\ldots, j;  \\
&d_{k+1}= d_{k}', \ \  \ \ \mbox{for} \  \ k=j+1,\ldots, g.
\end{align*}

\begin{align*}
&\mbox{For $k=1$, we have the identity $n_{1}'= a_1 n_j$,}\\
&\mbox{where $a_1$ is the value given in Proposition \ref{alternate};}\\
&\mbox{For $k=2,\ldots , j-1$, we have $n_k=n_k'$;}\\
&\mbox{For $k=j,\ldots , g-1$, we have $n_{k +1}=n_{k}'$.}
\end{align*}

\item Case \textbf{(IV)}:
\begin{align*}
&\mbox{$d_k' = d_{k+1}$, for $k=1,\ldots, g$ and $n_k' = n_{k+1}$, for $k=1,\ldots, g-1.$}
\end{align*}

\end{enumerate}
\end{lem}

\begin{proof}
Assume first we are in Case \textbf{(I)} ($i\geq 1$). The facts that  $d_k=d'_k$ for $k=1,\ldots , i$ and $d_{k+1}=d'_{k}$ for $k=j+1,\ldots , g$ follow by comparison of the $d(\Delta)$ and $d(\Delta')$-sequences.
For $k=i+1,\ldots , j$, we have
\begin{multline*}
 d_k= \mathrm{gcd}(\delta_0,\ldots , \delta_{i-1}, \delta_i=n_j\delta_j,\ldots , \delta_{k-1}) = n_j \cdot \mathrm{gcd}(\delta_0/n_j,\ldots , \delta_{i-1}/n_j, \delta_j,\ldots , \delta_{k-1}/n_j)\\
= n_j \cdot \mathrm{gcd}(\delta_0,\ldots , \delta_{i-1}, \delta_{j},\ldots , \delta_{k-1}) = n_j\cdot d'_k,
\end{multline*}
where the two last equalities follow from Proposition \ref{alternate}.

\medskip

Regarding the $n(\Delta)$ and $n(\Delta')$-sequences, it follows by definition that $n_k=n'_k$ for $k \in \{1,\ldots , i-1\} \cup \{ i+1,\ldots,j-1 \}$ and $n_{k+1}=n'_{k}$ for $k \in \{j+1,\ldots, g-1\}$. It only remains to study the cases $k=i$ and $k=j$. Both cases can be reasoned analogously. Let us see the case when $k=i$:
$$
n_i=\frac{d_i}{d_{i+1}}=\frac{\mathrm{gcd}(\delta_0,\ldots , \delta_{i-1})}{\mathrm{gcd}(\delta_0\ldots , \delta_i)}\cdot\frac{\mathrm{gcd}(\delta_0,\ldots , \delta_{i-1},\delta_j)}{\mathrm{gcd}(\delta_0,\ldots , \delta_{i-1},\delta_j)}=n_i' \cdot \frac{\mathrm{gcd}(\delta_0,\ldots , \delta_{i-1},\delta_j)}{\mathrm{gcd}(\delta_0,\ldots , \delta_{i-1},\delta_i)}=\frac{n_i'}{n_j},
$$
as desired.

The remaining cases (where $i=0$) follow similarly. Specifically, if $\delta_j > \delta_1$ (Case\textbf{(II)}), then $d_1'=\delta_j$, which implies $d_1=\delta_0=n_j \delta_j=n_j d_1'$. If $\delta_j<\delta_1$ (Case\textbf{(III)}), $d_1'=\delta_1, d_2=n_j d_2'$. By Proposition \ref{alternate}, $n_1'=n_j \delta_1/d_2 = a_1 n_j$.
\end{proof}


Our next result establishes a condition under which $\Delta'$ is a $\delta$-sequence.


\begin{theo}\label{prop:remove_delta}
Let $\Delta=(\delta_0,\ldots, \delta_g)$ be a $\delta$-sequence such that $\delta_i=n_j \delta_j$ for two indices $i$ and $j$ such that $0\leq i < j \leq g$. If $j <g$, assume further that $\delta_{j+1}<n_{j-1}\delta_{j-1}$. Then, the sequence
$$
\Delta '= (\delta_0', \delta_1', \ldots , \delta_{g-1}'),
$$
defined prior to Lemma \ref{lema:alpha}, is also a $\delta$-sequence.
\end{theo}

\begin{proof}
We must show that $\Delta'$ satisfies Conditions (a), (b) and (c) in Definition \ref{delta}. Condition (a) is deduced from Lemma \ref{lema:alpha} and the fact that $\Delta$ is a $\delta$-sequence. By applying the transposition $(i \ j)$, $\Delta'$ satisfies Condition (b) according to \cite[Proposition 50]{Shor}.

Finally, we verify  Condition (c), namely the inequality $\delta_k' < n_{k-1}'\delta_{k-1}'$ for $k\in [g-1]$. Assume first that we are in Case \textbf{(I)} as introduced before Lemma \ref{lema:alpha}. We demonstrate this inequality for the different values of $k$ by grouping them into the following subcases: 
\begin{itemize}
\item[(i)] $k=i$. 
\item[(ii)] $k=i+1$. 
\item[(iii)] $k=j$. 
\item[(iv)] $k\in \{1\ldots,g-1\}\setminus\{i,i+1,j\}$.
\end{itemize}
In Subcase (i), by Lemma \ref{lema:alpha} we must prove that $\delta_j<n_{i-1}\delta_{i-1}$. This inequality holds because $\delta_j<n_j\delta_j=\delta_i<n_{i-1}\delta_{i-1}$. In Subcase (ii), we should show that $\delta_{i+1}<n_jn_i\delta_j$ which is true because $n_j\delta_j=\delta_i$. In Subcase (iii), the inequality is $\delta_{j+1}<n_{j-1}\delta_{j-1}$, which holds by hypothesis. Only remains to prove Subcase (iv). By Lemma \ref{lema:alpha}, for $k=0, \ldots , i-2$ (only if $i>1$), we have that
$$
\delta_{k+1}'<n_{k}'\delta_{k}' \text{ is equivalent to } \delta_{k+1}<n_{k}\delta_{k}.
$$

In addition, for $i+1<k<j-1$ (only if $j-i>4$ ),
$$
\delta_k'<n_{k-1}'\delta_{k-1}' \text{ means }\delta_k<n_{k-1}\delta_{k-1}
$$
and, in a  similar way, for $j< k \leq g-1$ (only if $j<g-1$), we have that
$$
\delta_k'<n_{k-1}'\delta_{k-1}' \text{ is equivalent to } \delta_{k+1}<n_{k}\delta_{k}.
$$
Since $\Delta$ is a $\delta$-sequence, all the above inequalities in subcase (iv) are satisfied. This proves  the theorem for Case \textbf{(I)}.

The proofs of the remaining cases described before Lemma \ref{lema:alpha} are analogous with the exception of
 Case \textbf{(III)} regarding the inequality $\delta_2'<n_1'\delta_1'$.
In this specific situation, it suffices to prove that $\delta_2<a_1n_j\delta_j$, which is true since
$$
a_1n_j\delta_j=a_1\delta_0=n_1\delta_1
$$
by Proposition \ref{alternate}. This completes the proof.
\end{proof}

\begin{re}\label{remark:deltaseq_nominimally_gene}
Considering the $\delta$-sequences $\Delta_1$ and $\Delta_2$ in Example \ref{ex:new}, Theorem \ref{prop:remove_delta} allows us to deduce that $\Delta_1$ is a $\delta$-sequence from the fact that $\Delta_2$ is one.

Moreover, the requirement $\delta_{j+1}<n_{j-1}\delta_{j-1}$ if $j<g$ in Theorem \ref{prop:remove_delta} is essential. For example, consider the $\delta$-sequence
$$
\Delta=(13860, 12474, 2926, 4389, 1134, 8779 ).
$$
Although $12474=\delta_1=n_4\delta_4= 11\cdot 1134$, the sequence $\Delta'= (13860, 1134, 2926, 4389,$ $ 8779 )$ is not a $\delta$-sequence since $8779=\delta_5\nless n_3\delta_3=2\cdot 4389=8778$, thus the simplification fails. A similar situation happens in Remark \ref{nominimal}.

However, the mentioned requirement is unnecessary when $j=g$ and it always holds when $j=i+1$. The first case is clear, while for the second one, it is required that $\delta_{i+2} < n_i \delta_i$. This last condition holds because $\Delta$ is a $\delta$-sequence implying $\delta_{i+2} < n_{i+1} \delta_{i+1} = \delta_i$.
\end{re}

As stated at the beginning of this section, our objective is to start with a $\delta$-sequence $\Delta$ and its generated semigroup $S_\Delta$, and find, if possible, shorter $\delta$-sequences that generate $S_\Delta$. Geometrically, we seek  surfaces at infinity or curves with only one place at infinity whose $\delta$-semigroup or semigroup at infinity is $S=S_\Delta$ and that they are successively less intricate. In this sense, the most desirable surfaces or curves are those associated with primitive $\delta$-sequences of $S$. While Theorem \ref{prop:remove_delta} represents a significant toward this goal, it does not always determine primitive $\delta$-sequences of $S$ as illustrated in the following example.

\begin{example}
Consider the $\delta$-sequence
$$
(11760, 10290, 1575, 450, 2058, 7385)
$$
whose $n$-sequence is $(8,14,7,5,3)$. Here $$
10290=\delta_1=n_4\delta_4=5\cdot 2058 \text{ and }7\cdot 450 =n_3\delta_3<7385<5\cdot 2058=n_4\delta_4.
$$
This sequence does not satisfy the requirements in Theorem  \ref{prop:remove_delta}.

Indeed, the sequence $\Delta'$ in Theorem  \ref{prop:remove_delta},   
$(11760, 2058, 1575, 450, 7385)$, is not a $\delta$-se\-quen\-ce. However, a different rearrangement:  
$$(11760, 2058, 1575, 7385,450) $$
is, in fact, a MG$\delta$-sequence 
that generates the same semigroup $S_\Delta$.
\end{example}

Given a $\delta$-semigroup $S$ (generated by an initial $\delta$-sequence $\Delta$), $\delta$-sequences significantly shorter than $\Delta$ and, in most cases primitive $\delta$-sequence generating $S$, can be obtained under an algorithm involving rearrangements of $\delta$-sequences. This approach is computationally feasible for $\delta$-sequences of moderate size (both in terms of the length of the sequence and the magnitude of the integers involved). Utilizing Theorem \ref{prop:remove_delta} minimizes the number of required operations and generally yields substantially shorter $\delta$-sequences that generate $S$.

A sketch of the algorithm would be starting with a $\delta$-sequence $\Delta$ and, use Theorem \ref{prop:remove_delta} to delete multiples of other elements in $\Delta$. It can be applied successively for several multiples $\delta_i$ satisfying the condition for $\delta_{j+1}$ given in the theorem whenever one starts with the rightmost $\delta_i$. Then, one obtains a $\delta$-sequence $\Delta'$ that no longer satisfies the hypothesis of Theorem \ref{prop:remove_delta}. If $\Delta'$ contains  no multiples of its other elements, the algorithm terminates. Otherwise, in a second step, one can rearrange the elements following the last existing multiple an reapply Theorem \ref{prop:remove_delta} where possible. Otherwise (or alternatively), one can delete a multiple $\delta_j$ and test various permutations to find (if it exists) a shorter $\delta$-sequence $\Delta''$. If one cannot find $\Delta''$, $\delta_j$ cannot be erased of $\Delta$. Repeating this procedure (applying Theorem \ref{prop:remove_delta} whenever possible) yields a $\delta$-sequence of minimal length generated solely by elements in $\Delta$. The $\delta$-sequence obtained by this algorithm is primitive if it contains, at most, a multiple of some $\delta_j$ that cannot be erased. Otherwise, there may exist a shorter $\delta$-sequence containing multiples of elements $\delta_j$ that are not in $\Delta$. Although this latter scenario seems possible, the authors have been unable to find an example due to its computational complexity.

\section{Enlarging $\delta$-sequences for a fixed $\delta$-semigroup } \label{sec:Nested_delta_sequences}


In this section, we seek $\delta$-sequences that enlarge a given $\delta$-sequence $\Delta$ while generating the same semigroup as $\Delta$. Our enlargements preserve the ordering of the elements in $\Delta$. If desired, one may subsequently  perform rearrangements and verify whether the resulting sequences still satisfy the $\delta$-sequence condition. Geometrically, this corresponds to studying surfaces at infinity (or curves with only one place at infinity) that are more intricate than a previously fixed surface at infinity  (or curve with only one place at infinity), always sharing the same $\delta$-semigroup (or the same semigroup at infinity).

We begin by introducing  the concept of refinement of a $\delta$-sequence.

\begin{de}\label{25}
Let $\Delta = (\delta_i )_{i=0}^g$ and  $\Delta' = (\delta'_i )_{i=0}^{g'}$ be two $\delta$-sequences. The sequence $\Delta'$ is said to be a {\it refinement} of $\Delta$ if there exists a subset $\{ i_0 , i_1 , \ldots , i_{g} \}$ of $[0,g']$ with $i_0  <i_1 < \cdots <i_{g}$ such that $\delta'_{i_j}= \delta_j$ for every $j\in [0,g]$. 
The {\it order of refinement} of $\Delta'$ with respect to $\Delta$ is defined as the cardinality the set $\Delta' \setminus \Delta$.
\end{de}

For instance, the $\delta$-sequence $\Delta'=(108,72,24,54,26,13)$ is a refinement of order $2$ of  $\Delta=(108,24,54,13)$. Note that, in this case,  $\Delta$ can be obtained from $\Delta'$ from the procedure given in Theorem \ref{prop:remove_delta}.


In this section, we start with a $\delta$-sequence $\Delta$ and investigate the conditions under which elements can be added to $\Delta$ while preserving both the $\delta$-sequence property and the generated semigroup. It turns out that we can enlarge $\Delta$  in three different ways, which  can  be applied successively:

\begin{itemize}
\item[$\diamond$] Inserting a suitable sequence of positive integers between two generators.
\item[$\diamond$] Appending a suitable sequence of positive integers which are not multiples of $\delta_0$  before $\delta_0$.
\item[$\diamond$] Appending a suitable sequence of multiples of $\delta_0$  before $\delta_0$.
\end{itemize}

Let us study these operations. The following result characterizes the refinements of order one of a $\delta$-sequence that preserve the same semigroup. Given a positive integer $m$, $d(m)$ denotes some specific divisor of $m$.

\begin{theo}\label{thm:introducir_beta}
Let $\Delta=(\delta_i)_{i=0}^g$ be a $\delta$-sequence and consider  its decomposition as shown in Proposition \ref{alternate}. Let $\Delta'=(\delta_i')_{i=0}^{g+1}$ be a sequence in $\mathbb{N}_{>0}$. Denote by $(d_i')_{i=1}^{g+2}$ the $d(\Delta')$-sequence and by $(n_i')_{i=1}^{g+1}$ the $n(\Delta')$-sequence of $\Delta'$. Recall that $S_\Delta$ (respectively, $S_{\Delta'}$) denotes the semigroup generated by $\Delta$ (respectively, $\Delta'$) and, for $\Delta$ and $\Delta'$, keep the notation as in Proposition \ref{alternate}.
\begin{itemize}
\item[(a)] Assume that $\Delta'$ is of the form $(\beta,\delta_0,\delta_1,\ldots,\delta_g)$, then $\Delta'$ is a principal $\delta$-sequence such that $S_{\Delta'}=S_{\Delta}$ if and only if $\beta = \alpha \delta_0$ for some integer $\alpha>1$.

\item[(b)] Suppose again that $\Delta'$ is of the form $(\beta,\delta_0,\delta_1,\ldots,\delta_g)$, then $\Delta'$  is a non-principal $\delta$-sequence such that $S_{\Delta'}=S_{\Delta}$ if and only if there exists an index $j$, $1\leq j\leq g,$ such that the following conditions hold:
\begin{itemize}
\item[1.] The value $\beta$ can be expressed as
$$
\beta=n_{j+1}'\,\delta_{j+1}'=n_{j+1}'\,a_j\prod_{k=j+1}^g n_k.
$$

\item[2.] If $j=1$, then $n'_{k+1}=n_k,2\leq k\leq g.$ Additionally, $n_1=a_1'\cdot n_2'$ with $a_1',n_2'>1$ and $n_1'$ is a divisor of $a_1$ such that $a_1'<n_1'$ and $\gcd(a_1',n_1')=1$.

\item[3.] If $j\geq 2$, the value $a_j$ and the sequence $(n_k)_{k=1}^{j-1}$ satisfy
 $$\gcd\Big(a_j,\displaystyle\prod_{k=1}^{j-1}n_k\Big)>\gcd\Big(a_j,\displaystyle\prod_{k=2}^{j-1}n_k\Big)>\cdots > \gcd\left(a_j,n_{j-1}\right)>1.$$
Moreover, 
$1<d(a_j)=n_1'<a_j$,
$n_{k+1}'=d(n_{k})>1$ for $1\leq k \leq j,$ and there is an index $k_0$, $1 \leq k_0 \leq j$, such that $d(n_{k_0})\neq n_{k_0}$. In addition, $n_{k+1}'=n_k$ for $j+1\leq k \leq g$.

\item[4.] $\beta>\delta_0$ and $\delta_k<n_k'\delta_{k-1},$ for $2\leq k \leq j+1$.

\item[5.]  $n_1'\delta_0$ is a multiple of  $\beta$ and, for $1\leq k \leq j$,  $n_{k+1}'\delta_k$ belongs to the semigroup $\langle \beta,\delta_0,\ldots,\delta_{k-1}\rangle$.
\end{itemize}

\item[(c)] Finally, assume that $\Delta'$ has the form $(\delta_0,\delta_1,\ldots,\delta_{h-1},\beta,\delta_h,\ldots,\delta_g)$ for some index $h, 1\leq h\leq g.$ Then, $\Delta'$ is a $\delta$-sequence such that $S_{\Delta'}=S_{\Delta}$ if and only if there exists an index $j$, $h\leq j\leq g,$ such that the following conditions are satisfied:
\begin{itemize}
\item[1.] The value $\beta$ can be expressed as
$$
\beta=n_{j+1}'\delta_{j+1}'=n_{j+1}'\,a_{j}\prod_{k=j+1}^g n_k.
$$
Recall that we set $\prod_{k=g+1}^g n_k=1$.
\item[2.] The value $a_j$ and the sequence $(n_k)_{k=h}^{j}$ satisfy
$$
\prod_{k=h}^{j}n_k>\gcd\Big(a_j,\prod_{k=h}^{j-1}n_k\Big)>\gcd\Big(a_j,\prod_{k=h+1}^{j-1}n_k\Big)>\cdots > \gcd\left(a_j,n_{j-1}\right)>1.
$$

Moreover, $n_k'=n_k,$ $1\leq k \leq h-1$ (only if $h>1$);
$$1<n_h'=d\left(\prod_{k=h}^j n_k\right)<\prod_{k=h}^j n_k;$$  $n'_{k+1}=d(n_k)>1$, for $h \leq k \leq j$, and there is an index $h \leq k_0 \leq j$ such that $d(n_{k_0})\neq n_{k_0}$. In addition, $n_{k+1}'=n_k,j+1\leq k \leq g$.

\item[3.] $\delta_h<n_h'\beta<n_h' n_{h-1} \delta_{h-1}$ and $\delta_k<n_k'\delta_{k-1}$ for $h+1 \leq k \leq j+1$.
\item[4.] Finally,
\begin{equation*}
\begin{array}{c}
n_h'\beta\in\langle \delta_0,\delta_1,\ldots,\delta_{h-1} \rangle, \ n_{h+1}'\delta_h\in\langle \delta_0,\delta_1,\ldots,\delta_{h-1},\beta \rangle \text{ and } \\[2mm]
n_k'\delta_{k-1}\in\langle \delta_0,\delta_1,\ldots,\delta_{h-1}, \beta,\delta_h,\ldots,\delta_{k-2}\rangle\text{ for } h+2\leq k \leq j+1 \text{ (only if } j \geq h+1).
\end{array}
\end{equation*}
\end{itemize}
\end{itemize}
\end{theo}
%
%
%
\begin{proof}
We first show (a). Assume that $\beta = \alpha \delta_0$ for an integer $\alpha>1$.  Then, $$\Delta'= (\delta_0'=\alpha \delta_0, \delta_1'=\delta_0, \delta_2'=\delta_1, \ldots, \delta_{g+1}'=\delta_g)$$ is a $\delta$-sequence such that $S_{\Delta'} = S_\Delta$.  The facts that $d_i'=d_{i-1}$ for  $i=2,\ldots , g+2$, $n_1'=\alpha>1$ and $n_i'=n_{i-1}$ for $i=2,\ldots , g+1$ prove Items (a), (b) and (c) in Definition \ref{delta}. The converse is immediate.
\medskip

To prove (b), we initially assume that $\Delta'$ is a non-principal $\delta$-sequence such that $S_{\Delta'}=S_{\Delta}$ and prove that Conditions  1 through 5 in the statement hold. By Lemma \ref{lemma:prop3Sathaye}, there exists an index $j\in[g]$ such that $\beta=n_{j+1}'\delta_{j+1}'=n_{j+1}'\delta_{j}$. Moreover, by Proposition \ref{alternate}, the value $\beta$ can be written as $\beta= a_{j} n_{j+1}' \prod_{k=j+1}^g n_k$, satisfying Condition 1. Conditions 4 and 5  follow from the fact that $\Delta'$ is a $\delta$-sequence.
\medskip

Now we describe the $d(\Delta')$  and $n(\Delta')$-sequences of $\Delta'$ relative to the $d(\Delta)$ and $n(\Delta)$-sequences of $\Delta$. In both cases, $\beta$ is a multiple of $\delta_j;$ therefore
\begin{equation}\label{eq:corolario_beta_1}
d_{k+1}'=d_k, \text{ for } j+1\leq k\leq g+1, \text{ and }n_{k+1}'=n_k, \text{ for } j+1\leq k\leq g.
\end{equation}

Moreover, $d_1'=\beta$ and, since $\gcd(a_j,n_j)=1$ and $d_k=\prod_{\ell=k}^gn_\ell$, for $ 1 \leq k \leq g$, we have
\begin{multline*}
     d_{k+1}'=\gcd\big(n_{j+1}'\delta_j,\delta_1',\ldots,\delta_{k}'\big)=\gcd\big(n_{j+1}'\delta_j,
     \gcd\left(\delta_0,\ldots,\delta_{k-1}\right)\big)\\
     =\gcd\Big(n_{j+1}'a_j,\prod_{\ell=k}^{j}n_\ell\Big)\,\displaystyle\prod_{\ell=j+1}^gn_\ell,
\end{multline*}
where $k\in[0,j+1]$ and, by convention, we set $\prod_{\ell=j+1}^{j}n_\ell=1.$

In order to show that Condition 2 holds, assume $j=1$. From the above properties, it follows that $n'_1=d(a_1)$ and $n'_2=d(n_1)$. We prove the remaining properties asked in Condition 2. By Proposition \ref{alternate}, we have $a_1'<n_1'$ and $\gcd(a_1',n_1')=1.$ Also it holds that $\prod_{k=1}^gn_k=\delta_0=\delta_1'=a_1'\prod_{k=2}^gn_k'$, and then $n_1=a_1'\cdot n_2'$ by \eqref{eq:corolario_beta_1}. Both values $a_1'$ and $n_2'$ are strict bigger than 1 because $\Delta'$ is a non-principal $\delta$-sequence. Hence, the properties in Condition 2 are satisfied.

Now, assume $j>1$ to prove Condition 3. As a consequence of  \eqref{eq:corolario_beta_1},
\begin{equation}\label{eq:corolario_beta_2}
 n_{k+1}'=\dfrac{d_{k+1}'}{d_{k+2}'} = \dfrac{\gcd\Big(n_{j+1}'a_j,\displaystyle\prod_{\ell=k}^{j}n_\ell\Big)}{\gcd\Big(n_{j+1}'a_j,\displaystyle\prod_{\ell=k+1}^{j}n_\ell\Big)}=d(n_k),\; \text{ for }k\in[j].
\end{equation}
In particular,
$$
d_2'=\gcd (\beta,\delta_0)=\gcd\Big(a_j,\prod_{\ell=1}^{j-1}n_\ell \Big)n_{j+1}'\prod_{\ell=j+1}^g n_\ell,
$$
which yields
\begin{equation}\label{eq_daj}
n_{1}'=\dfrac{d_1'}{d_2'}=d(a_j).
\end{equation}
Since $\Delta'$ is a $\delta$-sequence, we have  $d(a_j)=n_{1}'>1$ and $d(n_k)=n_{k+1}'>1$ for $k\in[j]$, and then
$$
\gcd\Big(a_j,\prod_{\ell=1}^{j-1}n_\ell\Big)>\gcd\Big(a_j,\prod_{\ell=2}^{j-1}n_\ell\Big)>\cdots > \gcd (a_j,n_{j-1} )>1.
$$
Finally,  it holds that $1<d(a_j)=n_1'<a_j$ and, as $\Delta'$ is non-principal, $d(n_{k_0})\neq n_{k_0}$ for some $k_0\in \{1,\ldots, j\}$. Indeed,
by Proposition \ref{alternate} one can deduce that
$$
\delta_0'=\prod_{k=1}^{g+1}n_k'=d(a_j)\,\prod_{k=1}^jd(n_k)\prod_{k=j+1}^gn_k \text{ and }\delta_0=\prod_{k=1}^gn_k.
$$
Moreover, if $d(n_k)=n_k$ for every $k\in[j]$, then $\delta_0'$ is a multiple of $\delta_0$ which is a contradiction. Similarly, if $d(a_j)=a_j$, then $\gcd\Big(a_j,\prod_{\ell=1}^{j-1}n_\ell\Big)=1$ by  \eqref{eq_daj}, a contradiction. Thus, the properties in Condition (3) hold.
 \medskip

To conclude the proof of (b), it suffices to note that Conditions 1 through 5 of the statement combined with the fact that $\Delta$ is a $\delta$-sequence and Equalities \eqref{eq:corolario_beta_1} and \eqref{eq:corolario_beta_2} prove that $\Delta'$ is a $\delta$-sequence.
\medskip

The conditions in Item (c) can be proved in the same way as those in Item (b).  We are more specific to show Condition 2. It is clear that $d_k'=d_k$ for $k\in [h]$, $n_k'=n_k$ for $k \in [h-1]$, and
$$
n_h'=\dfrac{d_h'}{d_{h+1}'}=\dfrac{\gcd\left(\delta_0,\delta_1,\ldots,\delta_{h-1}\right)}{\gcd\big(\delta_0,\delta_1,\ldots,\delta_{h-1},n_{j+1}'\delta_j\big)}=\dfrac{\displaystyle\prod_{k=h}^g n_k}{\displaystyle\prod_{k=j+1}^g n_k \gcd\Big(a_j\,n_{j+1}',\prod_{k=h}^j n_k\Big)}=d\Big(\displaystyle\prod_{k=h}^jn_k\Big).
$$
Also our previous reasoning also shows that $n_{k_{0}+1}'=d(n_{k_{0}})\neq n_{k_{0}}$ for some $k_{0}\in\{h,\ldots , j\}$, and $n_h'\neq\prod_{k=h}^jn_k$. Indeed,
$$
\prod_{k=1}^gn_k=\delta_0=\delta_0'=\prod_{k=1}^{g+1}n_k'=n_h'\,\prod_{k=1}^{h-1}n_k \,\prod_{k=h}^{j}d(n_k)\prod_{k=j+1}^g n_k,
$$
and if either $d(n_k)= n_k$ for all $h\leq k\leq j,$ or $n_h'=\prod_{k=h}^jn_k$, then  it follows that $\delta_0=\alpha\delta_0'$ with $\alpha>1$, which is a contradiction.

Conditions 1, 2, 3 and 4 in Item (c) are sufficient to deduce that $\Delta'$ is a $\delta$-sequence such that $S_{\Delta'}=S_{\Delta}$.
\end{proof}

Let $\Delta_1$ and $\Delta_2$ be the $\delta$-sequences from Example \ref{ex:new}. The value $\beta=\delta_2^2$ we added to $\Delta_1$ to obtain $\Delta_2$ satisfies the conditions established in Item (c) of Theorem \ref{thm:introducir_beta}.

To provide a more complex illustration of $\delta$-sequences $\Delta$ and $\Delta'$ as described in Item (c) of Theorem \ref{thm:introducir_beta}, consider the $\delta$-sequence
$$
\Delta=(1944,162,81,24,64),
$$
whose $n$-sequence is $(12,2,27,3)$. The sequence $\Delta'=(1944,72,162,81,24,64)$ is a refinement of order one of $\Delta$ whose $n$-sequence is $(27,4,2,3,3)$. Here, $\beta=72$, $h=1$ and $j=3$. Using the notation from Theorem \ref{thm:introducir_beta}, where for suitable indices $n_i=n_i(\Delta)$ and $n'_i=n_i(\Delta')$, we have $n'_1=27 = d(\prod_{k=1}^{3} n_k) = d( 648)$, $n'_2= 4 = d (n_1) = d(12)$, $n'_3= 2 = d (n_2) = d(2)$, $n'_4 = 3 = d(n_3) = d(27)$ and $n'_5= n_4=3$.
\medskip

\begin{re}
Corollary \ref{c21} provides an upper bound for the length of those $\delta$-sequences with the same first element $\delta_0$. Given such a $\delta$-sequence, Theorem \ref{thm:introducir_beta} (c) characterizes the $\delta$-sequences that arise by inserting  a suitable positive integer between two of its generators. Note that the resulting $\delta$-sequence may be principal or not.
\end{re}

%
%

A $\delta$-sequence $\Delta = (\delta_i )_{i=0}^g$ admits finitely many refinements whose first element is $\delta_0$. We seek sets  of sequences $\left\{\Xi_r = ( \delta'_{r,1},\delta'_{r,2}, \ldots, \delta'_{r,s}) \right\}_{r=0}^{g-1}$ such that
$$
\Delta' = (\delta_0, \Xi_0,\delta_1, \Xi_1, \ldots, \delta_r, \Xi_r, \ldots, \delta_{g-1},\Xi_{g-1}, \delta_g),
$$
is a $\delta$-sequence, where the sequences $\Xi_r$ may be empty. A strategy to obtain sequences $\Xi_r$ as above is to suppose 
that the $d(\Delta)$-sequence is nested within the $d(\Delta')$-sequence though this is not always the case. Thus, if we choose the values $d_{r+1} = d_{r+1}(\Delta) $ and $d_{r+2} = d_{r+2}(\Delta)$ of $\Delta$, then the number of sequences of integer numbers satisfying the following successive strict divisibility conditions:
$$
d_{r+2} | c_0 | c_1 | \cdots | c_s | d_{r+1}
$$
is a bound on the number of sequences of values $d'_i = d_i (\Delta')$ corresponding to $\Delta'$ that one can insert between $d_{r+2}$ and $d_{r+1}$.

Under this assumption, for $j \in [s]$ the integers $\delta'_{r,j}$ must satisfy the equalities
$$
\gcd(\delta_0, \ldots, \delta_r, \delta'_{r,1} \ldots, \delta'_{r,j})=c_{s-j} \mbox{\ and\ } \gcd(\delta_0, \ldots, \delta_r, \delta'_{r,1}, \ldots, \delta'_{r,s}, \delta_{r+1})=d_{r+2}.
$$
Consequently, $\delta'_{r,j}$ cannot belong to the semigroup generated by the set $ \{ \delta_0, \ldots, \delta_r \}$ as that would imply $\gcd(\delta_0, \ldots, \delta_r, \delta'_{r,j}) = d_{r+1}$, which is impossible. Finally, for each choice $d_{r+2} | c_0 |  \cdots | c_s | d_{r+1}$, the values $\delta'_{r,j}$ must satisfy the growth conditions:
$$
\delta'_{r,1} < n_r \delta_r, \;\; \delta'_{r,2} < \frac{d_{r+1}}{c_s} \delta'_{r,1}, \; \ldots, \; \delta'_{r,s} <\frac{c_0}{d_{r+2}} \delta'_{r,s-1}.
$$

For example, considering the $\delta$-sequence $\Delta = (\delta_0=768, \delta_1=80, \delta_2=15)$, it holds that
$$
\Delta' = (\delta_0,\Xi_0,\delta_1,\delta_2)
$$
where $\Xi_0=(\delta_{01},\delta_{02},\delta_{03},\delta_{04})$ and $\delta_{01}=384$, $\delta_{02}=192$, $\delta_{03}=96,\delta_{04}=32$  is again a $\delta$-sequence.

\medskip

To conclude this section we show that the refinement procedure cannot be infinite with the exception of the successive repetition of the procedure in Item (a) of Theorem \ref{thm:introducir_beta}.

\medskip
We begin by defining the concept of a \emph{nested} family of $\delta$-sequences.

\begin{de}\label{de27}
A {\it nested} family of $\delta$-sequences is a (finite or infinite) sequence $\{\Delta_i\}_{i\geq 0}$ of $\delta$-sequences such that each $\Delta_{i+1}$ is a refinement of order 1 of  $\Delta_{i}$.
\end{de}



\begin{co}\label{theo:nested}
Let $\Delta=(\delta_0, \delta_1, \ldots, \delta_g)$ be a $\delta$-sequence. Then, there exists an infinite  nested family of $\delta$-sequences $\{\Delta_i\}_{i \geq 0}$ such that $\Delta_0 =\Delta$ and $S_{\Delta_i} = S_\Delta$  for all $i\geq 0$. 
\end{co}

\begin{proof}
If follows from Theorem \ref{thm:introducir_beta} (a).
\end{proof}

Corollary \ref{theo:nested} does not hold if prepending multiples of $\delta_0$ is disallowed.

\begin{co} \label{t23}
Let $\Delta := \Delta_0$ be a $\delta$-sequence. Then, there exist finitely many nested families of $\delta$-sequences $\mathcal{D}:=\{\Delta_i\}_{i \geq 0}$ such that $S_\Delta = S_{\Delta_i}$ for all $i\geq 0$ and $\delta_0 (\Delta_{i+1})$ is not a multiple of $\delta_0 (\Delta_i)$. Furthermore, the cardinality of any such family $\mathcal{D}$ is finite.
\end{co}

\begin{proof}

Let $\Delta = (\delta_0(\Delta), \delta_1(\Delta), \ldots, \delta_g(\Delta))$. By Corollary \ref{c21}, there are only finitely many $\delta$-sequences whose first element is $\delta_0(\Delta)$. Consequently, there are only finitely many families $\mathcal{D}$ as in the statement where every sequence $\Delta_i$ shares the same initial element. Furthermore, the cardinality of each such family $\mathcal{D}$ is finite. To prove the result, it suffices to show that the set of possible values for $\delta_0$ across all $\delta$-sequences in any $\mathcal{D}$ is finite. Let $\Delta_1$ be any refinement of order one of $\Delta$ which satisfies $\delta_0(\Delta_1) \neq \delta_0(\Delta)$. By Theorem \ref{thm:introducir_beta} (b), it must hold that:
\[
\delta_0(\Delta_1) = n_{j+1}(\Delta_1)\delta_{j+1}(\Delta_1) = d(n_j(\Delta))\delta_j(\Delta) \quad \text{for some } j > 0,
\]
where $d(n_j(\Delta))$ is a divisor of $n_j(\Delta)$ strictly greater than $1$. As a consequence, we have finitely many choices for $\delta_0(\Delta_1)$. Any refinement $\Delta_2$ of order one of $\Delta_1$ satisfies $\delta_0 (\Delta_2) = d(n_j(\Delta_1))\delta_j(\Delta_1)$, where $d(n_j(\Delta_1))$ is a divisor of $n_j(\Delta_1)$ greater than $1$. Furthermore, by Theorem \ref{thm:introducir_beta} (b), when considering the $n$-sequence of $\Delta_2$ for $j > 1$, some value $n_{k+1}(\Delta_2)$ must be a proper divisor of $n_k(\Delta_1)$ for $0 < k \leq j$; additionally, the newly introduced term $n_1(\Delta_2)$ is a proper divisor of $a_j(\Delta_1)$, where $a_j(\Delta_1)$ is the positive integer satisfying:
\[
\delta_j(\Delta_1) = a_j(\Delta_1) \prod_{\ell=j+1}^{g(\Delta_1)} n_\ell(\Delta_1).
\]
If $j = 1$, then the value $n_1(\Delta_2) > 1$ is a divisor of $a_1(\Delta_1) < n_1(\Delta_1)$. Consequently, the number of possible distinct values for $\delta_0$ in the $\delta$-sequences $\Delta_i$ is finite. Therefore, an infinite refinement as described in the statement is impossible, which concludes the proof.

\end{proof}

To illustrate Corollary \ref{t23}, consider $\Delta=\Delta_0=(375,135,102,283)$. Then we can prepend $\beta=675$ to obtain the $\delta$-sequence
$$
\Delta_1=(675,375,135,102,283),
$$
which obviously has the same semigroup as $\Delta_0$ since $675=5\cdot 135$.  Next, we can prepend $1125$ to the new sequence to get the $\delta$-sequence
$$
\Delta_2=(1125,675,375,135,102,283).
$$
Clearly $S_{\Delta_2}=S_{\Delta_1}=S_{\Delta}$ since $1125=3\cdot 375$. At this stage, the procedure terminates because we cannot add any element at the beginning since the $n$-sequence of $\Delta_2$ is $n(\Delta_2)=(5,3,5,5,3)$ and it consists only of prime numbers, cf. Theorem \ref{thm:introducir_beta} (b). The fact that the number of $\delta$-sequences starting with $375, 675$ and $1125$ is finite proves the existence of finitely many $\delta$-sequences with the same semigroup without adding multiples of $\delta_0$ at the beginning.
\medskip

We conclude this paper with an example that illustrates the relationship between $\delta$-sequences, their refinements and the extremal rays of the cone of curves for the associated surfaces at infinity.
\begin{example} \label{ejemplo57}
\vspace{1mm}

Consider the $\delta$-sequence $\Delta = \{12, 8, 9\}$  from Example \ref{ex:new} and 
the gluing of $\Delta$ and $T$ through $\alpha=15$ and $\beta=16$. This yields the MG$\delta$-sequence $$\Delta_1=(180,120,135,16).$$ By performing refinements as in Definition \ref{25}, we can determine families of $\delta$-sequences that correspond to surfaces at infinity whose $\delta$-semigroup remains $S_{\Delta_1}$. The structure and intersection products of the generators of the  cones of curves for these increasingly intricate obtained surfaces can be described using the methods showed in Example \ref{ex:new}. A representative family of such sequences is
\begin{equation*}
    \begin{array}{c}
       \mathcal{T}=\Big\{\Delta_1 = (180,120,135,16), \Delta_2 = (180,120,135,48,16), \\\Delta_3 = (180,120,135,80,16),
         \Delta_4 = (180,120,270,135,16), \\ \Delta_5 = (180,120,270,135,48,16),\Delta_6 = (180,120,270,135,80,16)\Big\}.
    \end{array}
\end{equation*}

Each $\delta$-sequence $\Delta_i$, $1 \leq i \leq 6$, consists of $g_i+1$ elements. The corresponding values for $n_{g_i}$ are $n_{g_1} = 15$, $n_{g_2} = 3$, $n_{g_3} = 5$, $n_{g_4} = 15 $, $n_{g_5} = 3$ and $n_{g_6} = 5$. Following the procedure described at the end of Subsection \ref{sect:curves}, the associated Puiseux exponents are:
\begin{align*}
& \mathcal{P}_{\Delta_1} = \big\{1, 405/60,539/15\big\},\\
&\mathcal{P}_{\Delta_2} = \big\{1, 405/60,507/15,227/3 \big\},\\
&\mathcal{P}_{\Delta_3} =  \big\{1,405/60,475/15,229/5 \big\},  \\
&\mathcal{P}_{\Delta_4} = \big\{1, 270/60,435/30,269/15 \big\}, \\
&\mathcal{P}_{\Delta_5} = \big\{1, 270/60,435/30,237/15,227/3 \big\}, \text{ and } \\
&\mathcal{P}_{\Delta_6} =  \big\{1, 270/60,435/30,205/15,229/5 \big\}.
\end{align*}

By appending a final element $\delta_{g_i +1}$ such that $0 \leq \delta_{g_i+1} \leq n_{g_i} \delta_{g_i}$, $1 \leq i \leq 6$, each $\delta$-sequence $\Delta_i$ defines a finite collection of $\delta$-sequences of type A. Each type A $\delta$-sequence determines the dual graph $\Gamma$ of a sequence $\pi$ of blowups at a configuration of infinitely near points over $\gp^2$, giving rise to a surface at infinity $X$. The graph $\Gamma$ reveals the structure and intersection products of the exceptional divisors whose classes generate the cone NE$(X)$. This graph can be computed from the corresponding sequence $\mathcal{P}_{\Delta_i}$ and the value $\delta_{g_i +1}$, where the latter determines the number of final free points in $\Gamma$.

As previously explained, each $\delta$-sequence $\Delta_i$ allows for the explicit calculation of the intersection matrix of the classes that generate the cone of curves for the corresponding surface $X$.
\end{example}

\section*{Acknowledgements}
The third author would like to thank the Department of Algebra, Analysis, Geometry and Topology, and IMUVa of Valladolid University for the support received when preparing this article.

\bibliographystyle{plain}
\bibliography{BIBLIO_paquete1}

\end{document}